\documentclass[11pt]{amsart}  

\usepackage[latin1]{inputenc}
\usepackage{amsmath}
\usepackage{amsfonts}
\usepackage{amssymb}
\usepackage{stmaryrd}
\usepackage{latexsym}
\usepackage{graphicx}
\usepackage{subfigure}
\usepackage{color}
\usepackage{mathtools}
\usepackage{verbatim}
\usepackage[all]{xy}
\usepackage{graphics}
\usepackage{pdfsync}
\usepackage[backref=page]{hyperref}
\usepackage{ytableau}
\usepackage{tikz}
\usepackage[normalem]{ulem}

\usepackage{enumitem}
\theoremstyle{plain}
\newtheorem{theorem}{Theorem}[section]
\newtheorem{proposition}[theorem]{Proposition}

\newtheorem{lemma}[theorem]{Lemma}

\theoremstyle{definition}
\newtheorem{definition}[theorem]{Definition}

\theoremstyle{remark}
\newtheorem{remark}[theorem]{Remark}
\newtheorem{example}[theorem]{Example}

\numberwithin{equation}{section}

\newcommand{\bC}{\mathbb{C}}
\newcommand{\bN}{\mathbb{N}}

\newcommand{\bR}{\mathbb{R}}
\newcommand{\bZ}{\mathbb{Z}}

\newcommand{\tcr}[1]{\textcolor{red}{#1}}
\newcommand{\tcb}[1]{\textcolor{blue}{#1}}

\newcommand{\sgrp}{\mathbb{S}}

\newcommand{\suchthat}{\;|\;}

\newlength\cellsize 
\savebox2{%
\begin{picture}(15,15)
\put(0,0){\line(1,0){15}}
\put(0,0){\line(0,1){15}}
\put(15,0){\line(0,1){15}}
\put(0,15){\line(1,0){15}}
\end{picture}}
\newcommand\cellify[1]{\def\thearg{#1}\def\nothing{}%
\ifx\thearg\nothing
\vrule width0pt height\cellsize depth0pt\else
\hbox to 0pt{\usebox2\hss}\fi%
\vbox to 15\unitlength{
\vss
\hbox to 15\unitlength{\hss$#1$\hss}
\vss}}
\newcommand\tableau[1]{\vtop{\let\\=\cr
\setlength\baselineskip{-16000pt}
\setlength\lineskiplimit{16000pt}
\setlength\lineskip{0pt}
\halign{&\cellify{##}\cr#1\crcr}}}
\savebox3{%
\begin{picture}(15,15)
\put(0,0){\line(1,0){15}}
\put(0,0){\line(0,1){15}}
\put(15,0){\line(0,1){15}}
\put(0,15){\line(1,0){15}}
\end{picture}}
\newcommand\expath[1]{%
\hbox to 0pt{\usebox3\hss}%
\vbox to 15\unitlength{
\vss
\hbox to 15\unitlength{\hss$#1$\hss}
\vss}}
\newcommand\bas[1]{\omit \vbox to \cellsize{ \vss \hbox to \cellsize{\hss$#1$\hss} \vss}}

\usepackage{tikz}
\usetikzlibrary{arrows,petri,topaths}
\usepackage{xcolor}
\usepackage{tikz-qtree}
\usetikzlibrary{positioning,arrows,calc,intersections,shapes,decorations}

\newcommand{\dsc}{\mathrm{des}} 

\newcommand{\ds}[1]{\big\langle {#1}\big\rangle}

\newcommand{\wc}[2]{\mathcal{W}_{#1}^{#2}}

\newcommand{\qint}[1]{(#1)_q} 
\newcommand{\qfact}[1]{(#1)_q!} 
\newcommand{\qbin}[2]{\genfrac(){0pt}{}{#1}{#2}_q}

\newcommand{\schub}[1]{\mathfrak{S}_{#1}}
\newcommand{\prob}{\mathbb{P}} 
\newcommand{\mca}[1]{\mathcal{#1}} 
\newcommand{\mcbk}{\mca{B}} 
\newcommand{\face}[1]{\mathsf{F}(#1)} 
\newcommand{\setfd}{\mathrm{FD}}

\newcommand{\inv}{\mathrm{inv}} 
\newcommand{\maj}{\mathrm{maj}} 

\newcommand{\rev}[1]{\mathrm{rev}(#1)} 
\newcommand{\mc}[1]{\mathcal{#1}} 
\newcommand{\val}{\mathrm{val}} 
\newcommand{\kly}{\mc{K}}
\newcommand{\supp}{\mathrm{Supp}}

\newcommand{\Perm}{\mathrm{Perm}}
\newcommand{\vol}{\mathrm{vol}}

\newcommand{\gz}{\mathrm{GT}} 
\newcommand{\mbf}[1]{\mathbf{#1}} 

\newcommand{\cube}{\mathbf{C}} 
\newcommand{\code}{\mathrm{code}} 

\newcommand{\psu}{\mathrm{psu}}

\begin{document}

\title{Remixed Eulerian numbers}
\author{Philippe Nadeau}
\address{Univ Lyon, Universit\'e Claude Bernard Lyon 1, CNRS UMR
5208, Institut Camille Jordan, 43 blvd. du 11 novembre 1918, F-69622 Villeurbanne cedex, France}
\email{\href{mailto:nadeau@math.univ-lyon1.fr}{nadeau@math.univ-lyon1.fr}}

\author{Vasu Tewari}
\address{Department of Mathematics, University of Hawaii at Manoa, Honolulu, HI 96822, USA}
\email{\href{mailto:vvtewari@math.hawaii.edu}{vvtewari@math.hawaii.edu}}
\thanks{P.~N is partially supported by the project ANR19-CE48-011-01. V.~T. acknowledges the support from Simons Collaboration Grant \#855592.}

\begin{abstract}
Remixed Eulerian numbers are a polynomial $q$-deformation of Postnikov's mixed Eulerian numbers. They arose naturally in previous work by the authors concerning the permutahedral variety and subsume well-known families of polynomials such as $q$-binomial coefficients and Garsia--Remmel's $q$-hit numbers.   
We study their combinatorics in more depth. As polynomials in $q$, they are shown to be symmetric and unimodal. By interpreting them as computing success probabilities in a simple probabilistic process we arrive at a combinatorial interpretation involving weighted trees. 
By decomposing the permutahedron into certain combinatorial cubes, we obtain a second combinatorial interpretation. At $q=1$, the former recovers Postnikov's interpretation whereas the latter recovers Liu's interpretation, both of which were obtained via methods different from ours.

\end{abstract}

\maketitle

\section{Introduction}
\label{sec:intro}

This article studies a large family of polynomials, the  \emph{re}mixed Eulerian numbers, which were introduced in a previous work of the authors~\cite{NT21}. 
The terminology follows that of Postnikov~\cite[\S 16]{Pos09} where mixed Eulerian numbers were introduced. We first recall their original geometric definition.

Throughout this article, $r$ is a positive integer. Consider real numbers $\lambda_1\geq \cdots \geq \lambda_{r+1}$ and let $\lambda\coloneqq (\lambda_1,\dots,\lambda_{r+1})$. The \emph{permutahedron} $\Perm(\lambda)$ is the convex hull of the points $\lambda_{\sigma}\coloneqq (\lambda_{\sigma(1)},\ldots,\lambda_{\sigma(r+1)})\in\bR^{r+1}$ where $\sigma$ is a permutation in the symmetric group $\sgrp_{r+1}$. It sits in the hyperplane $\{(z_i)\in \bR^{r+1}\suchthat z_1+\cdots+z_{r+1}=\lambda_1+\cdots+\lambda_{r+1}\}$. After projecting it to $\bR^r\times \{0\}$, one can compute its volume $\vol(\Perm(\lambda))$.
The latter is known to be a polynomial in the differences $\mu_i\coloneqq\lambda_i-\lambda_{i+1}$, homogeneous of degree $r$. 
 It can thus be written as
\begin{align}
\label{eq:volume_perm_mu_expansion}
\vol(\Perm(\lambda))=\sum_{{c}=(c_1,\ldots,c_r)}A_{c}\;\frac{\mu_1^{c_1}\cdots \mu_r^{c_r}}{c_1!\cdots c_r!}
\end{align}
with ${c}$ in $\wc{r}{}\coloneqq\{(c_1,\dots,c_r)\suchthat c_1+\cdots + c_r=r\}$.

\begin{definition}[{\cite[\S 16]{Pos09}}]
For $c\in\wc{r}{}$, $A_c$ is called a \emph{mixed Eulerian number}.
\end{definition}

 We now recall the definition of \emph{remixed Eulerian numbers} $A_c(q)$ introduced in \cite[\S 4.3]{NT21}, where it is also pointed out why $A_c(1)=A_c$.\smallskip
 
 Let  $\sgrp_{r+1}$ act on $\bC[q,x_1,\dots,x_{r+1}]$ by permuting the indices of the indeterminates $x_1$ through $x_{r+1}$. Consider the operator $\partial_{w_o}$ that acts on polynomials $f\in \bC[q,x_1,\dots,x_{r+1}]$ as
\begin{align}
\label{eq:max_divided_difference}
\partial_{w_o}(f)=\frac{1}{\prod_{1\leq i<j\leq r+1}(x_i-x_j)}\sum_{\sigma\in\sgrp_{r+1}}\epsilon(\sigma)\sigma(f),
\end{align}
where $\epsilon(\sigma)$ is the \emph{sign} of $\sigma$. Then $\partial_{w_o}(f)$ is a symmetric polynomial in $x_1,\ldots,x_{r+1}$. If $f$ is homogeneous of degree $d$ in $x_1,\ldots,x_{r+1}$, then $\partial_{w_o}(f)$ vanishes if $d<\binom{r+1}{2}$ and has degree $d-\binom{r+1}{2}$ in $x_1,\ldots,x_{r+1}$ otherwise. 

Given $f\in \bC[q,x_1,\dots,x_{r+1}]$, define the \emph{$q$-divided symmetrization operator} by
\begin{align}
\label{eq:qds}
\ds{f}_{r+1}^q=\partial_{w_o}\left(f\prod_{1\leq i<j-1\leq r}(qx_i-x_j)\right).
\end{align}
Now assume $f$ has total degree $r$ in $x_1,\ldots,x_{r+1}$. Then $f\prod_{1\leq i<j-1\leq r+1}(qx_i-x_j)$ has degree $\binom{r+1}{2}$, and thus $\ds{f}_{r+1}^q$ is a polynomial in $\bC[q]$ by the property of $\partial_{w_o}$ recalled above. In particular, for $c\in \wc{r}{}$ we consider the degree $r$ polynomial $y_c$ defined as:
\begin{align}
\label{eq:def_yc}
 y_c=x_1^{c_1}(x_1+x_2)^{c_2}\dots (x_1+\dots+x_r)^{c_r}.
\end{align}

\begin{definition}[Remixed Eulerian number $A_c(q)$]
\label{defi:remixed}
For $c\in \wc{r}{}$, the \emph{remixed Eulerian number} $A_{c}(q)\in \bC[q]$ is defined as
\begin{align*}
A_{c}(q)=\ds{y_c}_{r+1}^q.
\end{align*}
\end{definition}
The polynomials $A_{c}(q)$ may be equivalently defined in a number of other ways; see Section~\ref{sec:alternative_definitions}. 

\smallskip

Several properties of mixed Eulerian numbers $A_c$ were given by Postnikov in a long list~\cite[Theorem 16.3]{Pos09} that exhibits the rich combinatorics attached to them. The theorem was reproduced by Liu~\cite[Theorem 4.1]{Liu16} who used his combinatorial interpretation of $A_c$ to reprove several items on the list (and add some more). 

The next theorem shows that all of these properties of Postnikov's theorem $q$-deform nicely to $A_c(q)$. We have kept exactly the order in which Postnikov stated the properties in his statement. Where necessary, we appeal to standard notation (as can be found in \cite{St97,St99} for instance) for various permutation statistics.

\begin{theorem}
\label{th:q_list}
Let $c=(c_1,\dots,c_{r})\in \wc{r}{}$.
\begin{enumerate}[label=$(\arabic*)$]
\item \label{it1}$A_{c}(q)$ is a polynomial in $q$ with nonnegative integer coefficients.

\item \label{it2} $A_{(c_1,\dots,c_{r})}(q)=q^{\binom{r}{2}}A_{(c_r,\dots,c_{1})}(q^{-1})$.

\item \label{it3} $A_{(\ldots,0,r,0,\ldots)}(q)$ with $r$  in the $i$th position equals
\[
\sum_{\substack{\sigma\in \sgrp_r\\ \dsc(\sigma)=i-1}}q^{\maj(\sigma)}.
\]

\item \label{it4} $\sum_{c\in \wc{r}{}}\frac{A_{c_1,\dots,c_{r}}(q)}{c_1!\cdots c_r!}=\frac{\qfact{r}}{r!}(r+1)^{r-1}$.

\item \label{it5} $\sum_{c\in \wc{r}{}}A_{c_1,\dots,c_{r}}(q)=\qfact{r}\,\mathrm{Cat}_r$ where $\mathrm{Cat}_r=\frac{1}{r+1}\binom{2r}{r}$ is the $r$th Catalan number.

\item\label{it6} $A_{(\ldots,0,k,r-k,0,\ldots)}(q)$, with $k$ in $i$th position, equals
\[
\sum_{\substack{\sigma\in \sgrp_{r+1},\dsc(\sigma)=i\\ \sigma(r+1)=r+1-k}}q^{\maj(\sigma)-k}.
\]

\item \label{it7} $A_{1,\dots,1}(q)=\qfact{r}$.

\item \label{it8} $A_{k,0,\dots,0,r-k}(q)=q^{\binom{k}{2}} \qbin{r}{k}$.

\item \label{it9} Assume that $c$ satisfies $\sum_{i\leq j} c_i\geq j$ for $j=1,2,\dots,r$. Then one has 
\[A_c(q)=\qint{1}^{c_1}\qint{2}^{c_2}\cdots\qint{r}^{c_r}.\]
\end{enumerate}
\end{theorem}

Some of these properties were already given in~\cite{NT21}, while the others will be proved at various locations in the text: \ref{it1} is \cite[Proposition~5.4]{NT21}. \ref{it2} is Lemma~\ref{lem:reversal symmetry}. \ref{it3}, as well as \ref{it6}, \ref{it8} and \ref{it9}, are treated in Section~\ref{sec:examples}. For \ref{it4}, we refer to Section~\ref{sec:volumes}. \ref{it5} follows from the sum rule \cite[Proposition~5.4]{NT21}, as explained in Remark~\ref{rem:cyclic}. Finally, \ref{it7} is part of~\cite[Theorem 4.8]{NT21}.
\smallskip

Combinatorial interpretations for the case $q=1$ have been given in previous works: in~\cite[\S 17]{Pos09} Postnikov defined certain weighted trees to give a combinatorial interpretation for $A_c$. Another interpretation was given by Liu \cite{Liu16}, in terms of $C$-compatible permutations. As we will argue in this work, such permutations can be naturally seen as bilabeled trees with leaf labels $1,c_1+2,c_1+c_2+3,\dots,c_1+\cdots+c_r+r+1=2r+1$. Both these interpretations come from finding functional equations for the volume polynomial \eqref{eq:volume_perm_mu_expansion} and extracting coefficients.
  
  We will refine both these combinatorial interpretations by interpreting the powers of $q$ in each, by fairly different methods. These are stated in terms of certain families of trees that are described in detail in Sections~\ref{sec:postnikov_trees} and~\ref{sec:volumes}.
\begin{itemize}
\item  $A_{c}(q)$ is the total weight of all ``Postnikov trees'' with a fixed associated sequence $\mbf{i}$ of content $c$ (Theorem \ref{th:ac_via_postnikov_trees}).
\item $A_{c}(q)$ is the total weight of all ``bilabeled trees'' with leaf labels $1,c_1+2,c_1+c_2+3,\dots,c_1+\cdots+c_r+r+1=2r+1$ (Theorem \ref{th:remixed via trees}). 
\end{itemize}

\subsection*{Outline of the article}
We recall some alternative definitions of $A_c(q)$ in \S\ref{sec:alternative_definitions}; these have already appeared in \cite{NT21}. The last one is probabilistic in nature, and is used in \S\ref{sec:postnikov_trees} to give a first combinatorial interpretation of $A_c(q)$.
In \S\ref{sec:examples} we distinguish two large subfamilies of indices $c$ that give particularly nice polynomials $A_c(q)$.
In \S\ref{sec:degree} we show that the sequence of coefficients of $A_c(q)$ is always symmetric and unimodal. This requires the determination of the degree and valuation of $A_c(q)$.  
In \S\ref{sec:volumes} we link $A_c(q)$ with the original geometry for $q=1$, as the parameter $q$ is interpreted via a certain cubical dissection of the permutahedron. From this a second combinatorial interpretation of $A_c(q)$ follows; see \S\ref{sec:A_c bilabeled}.

\section{Alternative definitions of remixed Eulerian numbers}
\label{sec:alternative_definitions}

For $n\in \mathbb{Z}_{\geq 0}$ and $0\leq k\leq n$, set
\begin{align}\label{eq:q-defs}
\qint{n}\coloneqq \frac{q^n-1}{q-1}=1+q+\cdots+q^{n-1};\quad
\qfact{n}\coloneqq \prod_{1\leq i\leq n} \qint{i};\quad \qbin{n}{k}\coloneqq\frac{\qfact{n}}{\qfact{k}\qfact{n-k}}.
\end{align}
These are the $q$-integers, $q$-factorials and $q$-binomial coefficients. Given integers $a\leq b$, we denote the interval $\{a,a+1,\dots,b\}$ by $[a,b]$. If $a=1$, we often shorten this to $[b]$.
For any undefined combinatorial terminology, we refer the reader to standard texts such as \cite{St97,St99}.

In contrast to the computational perspective provided in Definition~\ref{defi:remixed}, we offer three more perspectives that may be treated as alternative definitions. 

The reader will note the similarity between the perspectives that follow: they are indeed easily seen to be equivalent. In contrast, the fact that any of these definitions is equivalent to Definition~\ref{defi:remixed} is not  obvious, and this was a key result in~\cite{NT21}.

\subsection{Coefficients in the Klyachko algebra}
\label{sub:coeffs_klyachko}
 
The \emph{$q$-Klyachko algebra} $\kly$ is the commutative algebra over $\bC(q)$ on the generators $\{u_i\suchthat i\in \mathbb{Z}\}$ subject to the following relations for all $i\in\bZ$:
\begin{align*}
(q+1)u_i^2=qu_iu_{i-1}+u_iu_{i+1}.
\end{align*}
For a finite subset $I\subset \bZ$, let $u_I\coloneqq \prod_{i\in I}u_i$. By \cite[Proposition 3.9]{NT21}, the set $\mcbk$ of such squarefree monomials forms a linear basis for $\kly$. Given $c\in\mca{W}_r$, we have
\begin{align}
\label{eq:alg perspective}
A_c(q)=\qfact{r}\times\text{ coefficient of } u_{[r]} \text{ in the expansion of } u_1^{c_1}\dots u_r^{c_r} \text{ in } \mcbk.
\end{align}

\subsection{Recurrence relation}
\label{sub:recurrence_relation}

The remixed Eulerian number $A_c(q)$ for  $c=(c_1,\dots,c_r)\in \wc{r}{}$ is the unique polynomial satisfying the initial condition $A_{(1^{r})}(q)=\qfact{r}$ and the relation
\begin{equation}
\label{eq:relation_A}
(q+1)A_c(q)=q A_{(\dots,c_{i-1}+1,c_{i}-1,\dots)}(q)+A_{(\dots,c_{i}-1,c_{i+1}+1,\dots)}(q) \text{ for any } i \text{ satisfying } c_i\geq 2.
\end{equation}
On the right hand side we ignore the ill-defined terms in the case $i=1$ or $i=r$.

A more efficient recurrence based approach is as follows. The initial condition continues to be $A_{(1^r)}(q)=\qfact{r}$.
Otherwise, consider $i$  such that $c_i\geq 2$.
Let $\supp(c)$ denote the \emph{support} of $c$, i.e. the set of  all indices $j$ such that $c_j>0$.
 Let $[a,b]$ be the maximal interval in $\supp(c)$ containing $i$.  
We let $L_i(c)$ (resp. $ R_i(c)$) denote the composition obtained by decrementing $c_i$ by $1$ and incrementing $c_{a-1}$ (resp. $c_{b+1}$) by $1$. Then we have
 \begin{equation}
\label{eq:recurrence_A}
\qint{b-a+2}A_{c}(q)=q^{i-a+1}\qint{b-i+1}A_{L_i(c)}(q)+\qint{i-a+1}A_{R_i(c)}(q).
\end{equation}
Again, if $a=1$ (resp. $b=r$) then $L_i(c)$  (resp. $R_i(c)$) is not well defined, and the corresponding remixed Eulerians are $0$.

\subsection{Probabilistic interpretation}
\label{sub:proba_process}

 Consider the integer line $\bZ$ as a set of \emph{sites}.
A \emph{configuration} is an $\bN$-vector $c=(c_i)_{i\in \bZ}$ with $\sum_ic_i<\infty$, which we visualize as a finite set of particles with $c_i$ particles stacked at site $i$. \emph{Stable} configurations are those for which $c_i\leq 1$ for all $i$. They are identified with finite subsets of $\bZ$ via their support.

Given \emph{jump probabilities} $q_L,q_R\geq 0$ satisfying $q_L+q_R=1$, consider the following process with state space the set of configurations: Suppose we are in a configuration $c$. If $c$ is stable, the process stops. Otherwise, pick any $i$ such that $c_i>1$ and move the top particle at site $i$ to the top of site $i-1$  (resp. $i+1$) with probability $q_L$  (resp. $q_R$). 
The process ends in a stable configuration with probability $1$. Furthermore, it is known that the probability of ending in a particular stable configuration does not depend  on the choice of site $i$ at each step.

For a finite subset $I\subset\bZ$, we can then define $\prob_c(I)$ to be the probability that the process starting at configuration $c$ ends in the stable configuration given by $I$. Assume $q_R>0$ and let $q\coloneqq q_L/q_R$.
By \cite[Proposition 5.1]{NT21}, for $c\in \wc{r}{}$ we have
\begin{equation}
\label{eq:Ac_proba}
\prob_c([r])=\frac{A_c(q)}{\qfact{r}}.
\end{equation}

Note that we have then $q_L=\frac{q}{1+q}$ and $q_R=\frac{1}{1+q}$.

\begin{remark}
\label{rem:cyclic}
Using this interpretation one has the following ``cyclic sum rule''~\cite[Proposition 5.3]{NT21}: Given $c=(c_1,\dots,c_{r})\in \wc{r}{}$, let $\mathsf{Cyc}(c)$ be the set of all $c'\in \wc{r}{}$ such that $(c',0)$ is a cyclic rotation of $(c,0)$. Then
\begin{align}
\label{eq:sum_rule}
\sum_{c'\in \mathsf{Cyc}(c)}A_{c'}(q)=\qfact{r}.
\end{align}
There are $\mathrm{Cat}_r$ sets of the form $\mathsf{Cyc}(c)$, so that summing \eqref{eq:sum_rule} over all of them proves Theorem~\ref{th:q_list}\ref{it8}.
\end{remark}

We record a slightly different way to think about the previous process as it will be particularly helpful.\smallskip

\textbf{Sequential process:}  
Fix any word $\mathbf{i}=(i_1,\ldots,i_r)\in [r]^r$.  
Starting with the empty configuration, drop particles one at a time at sites $i_1,i_{2},\ldots, i_r$, \emph{and stabilize at every step}. 
Each such step involves the particle either landing on an interval of occupied sites in the current stable configuration, and then proceeding to exit either to the left or to the right, or landing on an unoccupied site in which case it stabilizes immediately. We denote by $\prob^{\mathbf{i}}(I)$ the probability to end up with the stable set $I$.

 Let $\mathrm{cont}(\mathbf{i})=(c_1,\dots,c_r)\in \wc{r}{}$ where $c_j$ is the the number of instances of $j$ in $\mathbf{i}$ for $j\in [r]$. Then we have 
 \begin{equation}\prob^{\mathbf{i}}(I)=\prob_{\mathrm{cont}(\mathbf{i})}(I)\end{equation}
  and thus by~\eqref{eq:Ac_proba},
\begin{equation}
\label{eq:Ac_probabis}
\prob^{\mathbf{i}}([r])=\frac{A_{\mathrm{cont}(\mathbf{i})}(q)}{\qfact{r}}.
\end{equation}

\section{Combinatorial interpretation via Postnikov trees}
\label{sec:postnikov_trees}

Postnikov showed that mixed Eulerian numbers $A_c$ enumerate a certain family of trees\cite{Pos09}.  His proof uses an equation for the volume of the permutahedron $\vol(\Perm(\lambda))$, proved purely geometrically (and valid for any ``root system''). Using the expansion~\eqref{eq:volume_perm_mu_expansion} and differentiation, he then obtains a combinatorial interpretation of the numbers $A_c$ as enumerating certain weighted trees. Here we will define a $q$-deformation of these weights, based on the probability process, that turns out to give a combinatorial interpretation of $A_c(q)$. Postnikov's trees can then be reinterpreted naturally as recording possible  \emph{histories} of the process.

\subsection{A recursive formula for $\prob^{\mathbf{i}}([a,b])$}

Consider $\mathbf{i}=i_1\cdots i_r$ and let us compute $\prob^{\mathbf{i}}([a,b])$, where we assume that $b-a+1=r$ and $i_j\in [a,b]$ for all $j$ since $\prob^{\mathbf{i}}([a,b])=0$ otherwise. 

Let us condition on the stable set obtained just before dropping the last particle $i_r$. This set is necessarily of the form $[a,b]\setminus\{j\}=[a,j-1]\sqcup [j+1,b]$ for $j\in[a,b]$ in order to have $\prob^{\mathbf{i}}([a,b])\neq 0$. We thus get
\begin{equation}
\label{eq:proba_conditioning}
\prob^{\mathbf{i}}([a,b])=\sum_{j=1}^r\prob^{i_1\cdots i_{r-1}}([a,b]\setminus\{j\})\,\prob^{i_r}([a,b],j),
\end{equation}
 where $\prob^i([a,b],j)$ is the probability to reach the stable set $[a,b]$ from $[a,b]\setminus\{j\}$ after dropping a particle at the site $i\in [a,b]$ and stabilizing.
 
The probability $\prob^{i_1\cdots i_{r-1}}([a,b]\setminus\{j\})$ is clearly zero unless \begin{enumerate}[label=(*)]\item \label{asterisk} there are $j-a$ indices $t$ such that $i_t\in[a,j-1]$ and $b-j$ indices $t$ such that $i_t\in[j+1,b]$. 
\end{enumerate}
Indeed the site $j$ is empty at all times before the last step, so the particles that are dropped on either side of it stay on that side during the process. Assuming~\ref{asterisk} is satisfied, let $\mathbf{i'},\mathbf{i''}$ be the two subsequences of $\mathbf{i}$ consisting of $i_t<j$ and $i_t>j$ respectively. Then we have
 \begin{align}
 \label{eq:proba_conditioning_bis}
 \prob^{i_1\cdots i_{r-1}}([a,b]\setminus\{j\})=\prob^{\mathbf{i'}}([a,j-1])\;\prob^{\mathbf{i''}}([j+1,b]).
 \end{align}
 
The other factor $\prob^{i_r}([a,b],j)$ in~\eqref{eq:proba_conditioning} is also straightforward to compute: if $i_r>j$, then the particle must exit to the left of the interval $[j+1,b]$ while if $i_r\leq j$, it must exit to the right of the interval $[a,j-1]$. These are the well-known exit probabilities of a biased discrete random walk on an interval. We thus obtain explicitly $\prob^{i_r}([a,b],j)=\mathrm{wt}_q([a,b],j,i_r)$ where for any $i\in [a,b]$,
\begin{align}
\label{eq:proba_conditioning_ter}
\mathrm{wt}_q([a,b],j,i)=\left\lbrace \begin{array}{ll}\frac{\qint{b-i+1}}{\qint{b-j+1}}q^{(i-j)} & i > j, \\ 
\frac{\qint{i-a+1}}{\qint{j-a+1}} & i \leq j, \end{array}\right.
\end{align}
The reader may recognize these as $q$-deformations of the weights $\mathrm{wt}(i,j)$ in \cite[Equation 17.1]{Pos09}, which can therefore be interpreted as exit probabilities in the symmetric case $q=1$.

Substituting \eqref{eq:proba_conditioning_bis},\eqref{eq:proba_conditioning_ter} in \eqref{eq:proba_conditioning} gives a recursive way to compute $\prob^{\mathbf{i}}([a,b])$:
\begin{align}
\label{eq:proba_conditioning_recursive}
\prob^{\mathbf{i}}([a,b])=\sum_{j\text{ s.t.~\ref{asterisk} holds}}\prob^{\mathbf{i'}}([a,j-1])\,\prob^{\mathbf{i''}}([j+1,b])\,\mathrm{wt}_q([a,b],j,i_r).
\end{align}
The initial condition is simply $\prob^{\epsilon}(\emptyset)=1$ where $\epsilon$ is the empty word.

\subsection{Postnikov trees}

Recall that the \emph{binary search labeling} of a binary tree $T$ with $r$ nodes is the canonical, standard labeling of nodes with $[r]$ given recursively by traversing the left subtree first, then the root, then the right subtree. This is illustrated in Figure~\ref{fig:postnikov_tree} by the labels inside the nodes. Let the \emph{bs-label} of a node be this label $bs(v)\coloneqq j\in[r]$.

Given a node $v$ in $T$, let $[l_v,r_v]$ refer to the set of bs-labels of its descendants in $T$. A coloring $f:\mathrm{Nodes}(T)\mapsto \bZ$ is {\em compatible} if for all $v$, we have $f(v)\in [l_v,r_v]$. The weight $\mathrm{wt}(T,f)$ of such a tree is then defined as \[\mathrm{wt}(T,f)=\prod_{v\in T}\mathrm{wt}_q([l_v, r_v],v,f(v)).\]

We call a labeled tree \emph{decreasing} if it has a standard labeling such that the label of a node is larger than the labels of all its descendants. On the leftmost panel in Figure~\ref{fig:postnikov_tree}, the exterior labels (in blue) give a decreasing labeling of the underlying tree.

\begin{definition}
Given $\mathbf{i}=i_1\cdots i_r\in [r]^r$, a tree $T$ is \emph{$\mathbf{i}$-compatible} if it has a (necessarily unique) decreasing labeling such that $v\mapsto i_{\mathrm{dec}(v)}$ is a compatible labeling. The weight $\mathrm{wt}(T,\mathbf{i})$ is the weight of  this compatible labeling.
\end{definition}
\noindent For instance, the tree $T$ in Figure~\ref{fig:postnikov_tree} is $\mbf{i}$-compatible where $\mbf{i}=34717843$. Furthermore we have 
\[
\mathrm{wt}(T,\mbf{i})=\frac{\qint{1}
}{\qint{1}}\cdot \frac{\qint{1}
}{\qint{1}} \cdot \frac{\qint{1}
}{\qint{1}}\cdot \frac{\qint{1}
}{\qint{1}}\cdot \frac{\qint{1}
}{\qint{2}}\cdot q^2\frac{\qint{1}
}{\qint{3}}\cdot q^2\frac{\qint{1}
}{\qint{3}}\cdot\frac{\qint{3}
}{\qint{5}}=\frac{q^4}{\qint{2}\qint{3}\qint{5}}.
\]
The order of the terms on the right hand side is obtained by considering the weights of nodes $v$ encountered according to the decreasing labeling.

\begin{remark}
Postnikov uses {\em increasing trees} instead of decreasing ones. Our notion of $\mathbf{i}$-compatibility corresponds to $\rev{\mathbf{i}}$-compatibility in his sense. This choice changes nothing of the underlying combinatorics, and was made by us because it matches more naturally with the probabilistic process. 
In addition his trees record the decreasing labeling together with the $\mathbf{i}$-labeling. As we pointed out in the previous definition, given $(T,\mathbf{i})$, such a decreasing labeling is in fact necessarily unique and can thus safely be removed from the definition.
\end{remark}

\begin{figure}[!ht]
\includegraphics[width=0.92\textwidth]{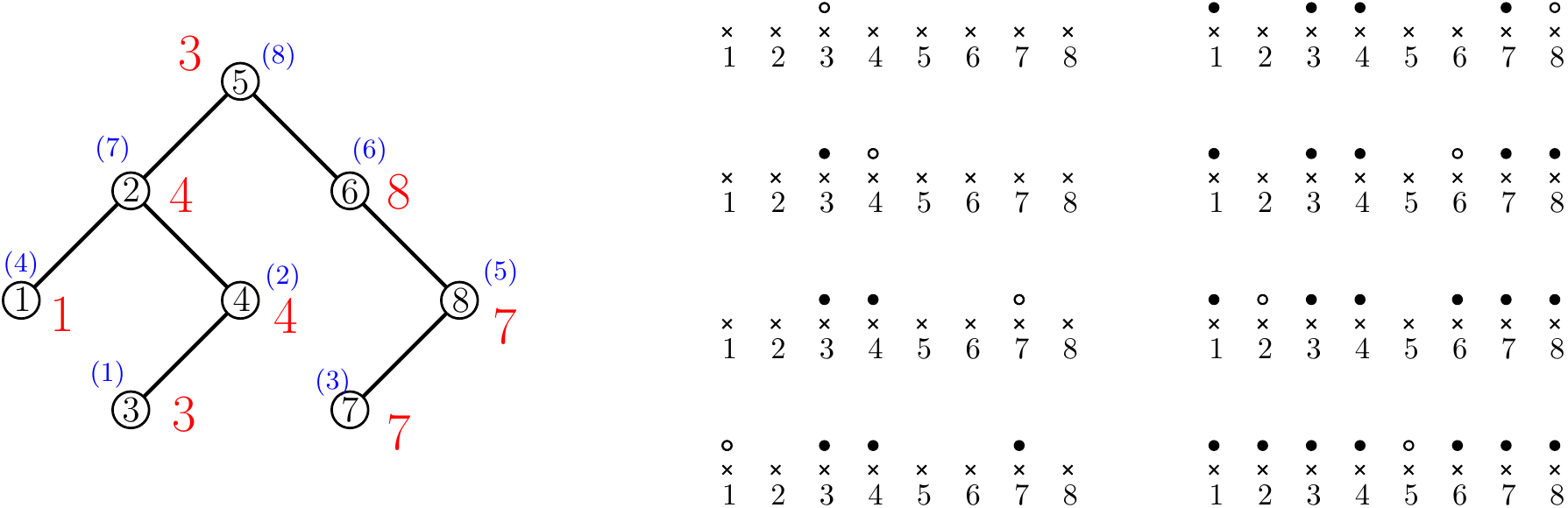}
\caption{With \tcr{$\mathbf{i}=34717843$}, a Postnikov tree is represented on the left: the canonical bs-labeling is pictured inside the nodes, while the unique \tcb{decreasing labeling} is between brackets in blue.\\ The associated history of the sequential process is illustrated in the second and third panels, reading each panel from top to bottom.}
\label{fig:postnikov_tree}
\end{figure}

Let $\mc{P}(\mathbf{i})$ be the set of $\mathbf{i}$-compatible trees. These are referred to as Postnikov trees in the introduction, with content $\mathrm{cont}(\mbf{i})$ determined by the multiplicity of each letter in $\mbf{i}$.

\begin{theorem}
For any $\mathbf{i}\in [r]^r$, we have
\begin{align*}
\prob^{\mathbf{i}}([r])=\sum_{T\in\mc{P}(\mathbf{i})}\mathrm{wt}(T,\mathbf{i}).
\end{align*}
\end{theorem}

\begin{proof}
This follows from the recurrence \eqref{eq:proba_conditioning_recursive}.  Let us specify that recurrence to the case $a=1,b=r$, and write $\mathbf{i^2}$ for the word $\mathbf{i''}$ where all letters are decreased by $j$, so that they are between $1$ and $r-j$. We obtain the recurrence
\[
\prob^{\mathbf{i}}([r])=\sum_{j\text{ s.t. ~\ref{asterisk}}}\prob^{\mathbf{i'}}([j-1])\,\prob^{\mathbf{i^2}}([r-j])\,\mathrm{wt}_q([r],j,i_r).
\]

We need to show  that $\sum_{T\in\mc{P}(\mathbf{i})}\mathrm{wt}(T,\mathbf{i})$  also satisfies such a recurrence.
 Given a tree in $\mc{P}(\mathbf{i})$, let $j-1$ and $r-j$ be the sizes of its left and right subtrees respectively. The root has necessarily label $i_r$ and thus weight $\mathrm{wt}_q([r],j,i_r)$. The compatibility condition imposes that all labels of the left subtree $T'$ are between $1$ and $j-1$, while all labels on the right subtree $T''$ are between $j+1$ and $r$. This corresponds precisely to the condition~\ref{asterisk} that has to be satisfied, and we let $\mathbf{i'},\mathbf{i''}$ be the subsequences of $\mathbf{i}$ corresponding to $T'$ and $T''$. Then $T'$ is in $\mc{P}(\mathbf{i'})$, while $T''$ is in  $\mc{P}(\mathbf{i''})$. Putting things together, we obtain the desired result.
\end{proof}

A direct byproduct of the proof is that trees in $\mc{P}(\mathbf{i})$ correspond precisely to a ``history'' of the sequential process started with $\mathbf{i}$. The decreasing tree encodes the \emph{filling order} in which sites get occupied along the process: The $k$th site to be occupied is $j=bs(v)$ where $v$ is the node with $\mathrm{dec}(v)=k$. This is nothing but the standard bijection between decreasing trees and permutations, which we recall at the beginning of Section~\ref{sub:cube_decomposition}. Such a filling order is possible with $\mathbf{i}$ as initial word precisely when the corresponding decreasing tree is $\mathbf{i}$-compatible. In that case, the weight of the tree is the probability of that ordering.

\begin{example}
For ease of comparison with Postnikov's interpretation, we revisit \cite[Example 17.6]{Pos09}. 
Consider $\mathbf{i}=34717843$. The process of dropping particles and stabilizing at each step is depicted in Figure~\ref{fig:postnikov_tree} on the right, and the corresponding Postnikov tree is on the left. The labels inside record the canonical binary search labeling, whereas the decreasing labeling is on the outside in parentheses.
\end{example}

Equation~\eqref{eq:Ac_probabis} now immediately yields the following $q$-analogue of \cite[Theorem 17.7]{Pos09}:

\begin{theorem}
\label{th:ac_via_postnikov_trees}
For any $c\in\wc{r}{}$ and any $\mathbf{i}\in [r]^r$ such that $\mathrm{cont}(\mathbf{i})=c$, we have
\begin{align*}
A_{c}(q)=\sum_{T\in\mc{P}(\mathbf{i})} \qfact{r}\,\mathrm{wt}(T,\mathbf{i}).
\end{align*}
\end{theorem}

\begin{example}
Consider $\mbf{i}=2244$. Then $\mathrm{cont}(\mbf{i})=(0,2,0,2)$. 
Figure~\ref{fig:22_44} shows the relevant Postnikov trees.
\begin{figure}[!ht]
\includegraphics[width=0.5\textwidth]{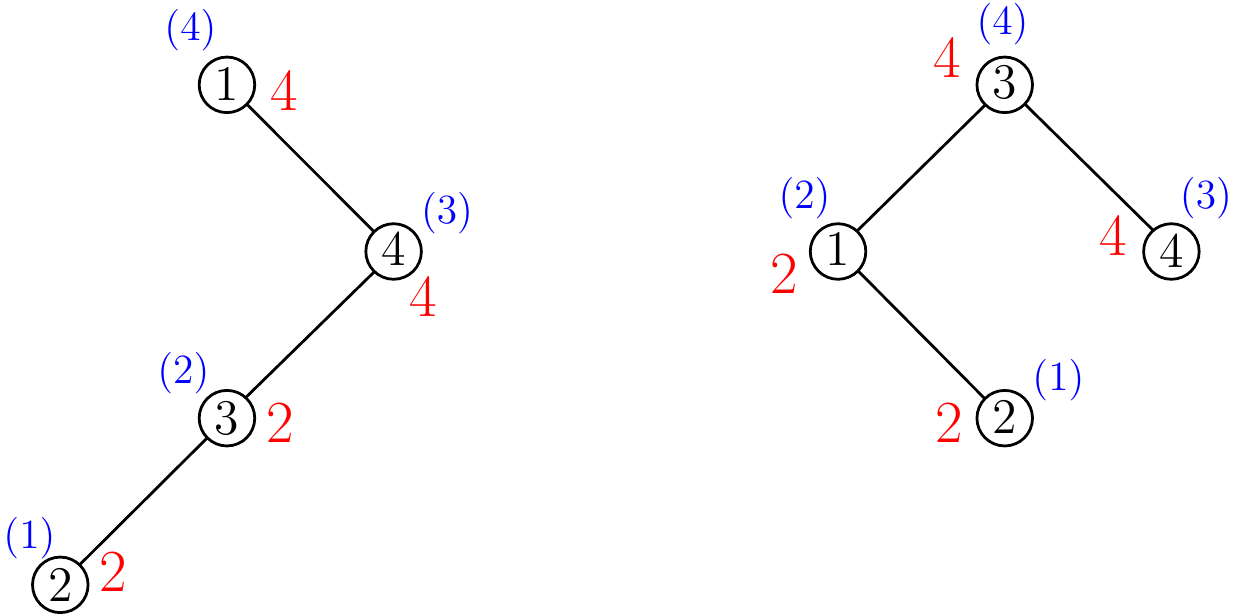}
\label{fig:22_44}
\caption{The two Postnikov trees for $\mbf{i}=2244$.}
\end{figure}

Theorem~\ref{th:ac_via_postnikov_trees} now tells us that
\[
A_{(0,2,0,2)}(q)=\qfact{4}\cdot\frac{1}{\qint{2}}\cdot q^3\frac{\qint{1}}{\qint{4}}+ \qfact{4}\cdot q\frac{1}{\qint{2}}\cdot q\frac{\qint{1}}{\qint{2}}=q^2(1+q+q^2)^2.
\]
\end{example}

\begin{remark}
Note that each of the two summands above belong to $\bN[q]$. This is the case in general indeed. Keeping with the theme, one simply needs to `$q$-ify' \cite[Lemma 17.5]{Pos09}. Consider $\qfact{r}$ times all denominators in $\mathrm{wt}(T,\mbf{i})$ times the power of $q$ accumulated in the numerator.
This rational function is a polynomial which tracks $q^{\ell(w)}$ over the following set of permutations $w$, contingent on $T$ and $\mbf{i}$ naturally. 
As in loc. cit., consider labelings of the underlying tree $T$ with permutations $w\in \sgrp_r$ such that if we are in the first case (resp. second case) of \eqref{eq:proba_conditioning_ter} for some node $j$, then  $w(j)$ exceeds the labels $w(k)$ for all $k$ in its right (resp. left) subtree.\footnote{There is a minor typo in the statement of \cite[Lemma 17.5]{Pos09}: the word `branch' should be replaced by `tree'.}
\end{remark}

\begin{remark}
Since different choices of $\mathbf{i}$ with same $\mathrm{cont}(\mathbf{i})$ determined the same remixed Eulerian number, smart choices can optimize things a bit. For instance, starting the sequence $\mathbf{i}$ with all elements of $\supp(c)$ (in any order) gives simpler trees in general. This corresponds of course to starting the particle process by dropping one particle at each site of the support.
\end{remark}

\section{Two special subfamilies}
\label{sec:examples}

We consider two large subfamilies of $A_c(q)$ by restricting the possible indices  in $\supp(c)$. These families encompass all special cases listed in Theorem~\ref{th:q_list}.
 They have specific properties that make them particularly nice from an enumerative standpoint: elements of the first family have elementary product formulas that generalize $q$-binomial coefficients, cf. Proposition~\ref{prop:product_formula}. Elements of the second subfamily have certain generating functions whose coefficients have simple product formulas, cf.~\eqref{eq:qhit_gf}, and coincide with the family of {\em $q$-hit polynomials}. 

\subsection{An extension of $q$-binomial coefficients.}
\label{sub:qbinomial}

\begin{definition}
\label{defi:ebr}
We define $EB_r\subset \wc{r}{}$ to be the set of $c=(c_1,\dots,c_{r})$ such that there exists a $k\in\{0,\ldots,r\}$ satisfying $\sum_{i\leq j}c_i\geq j$ for $j=1,\ldots,k$ and  $\sum_{i\leq j}c_{r+1-i}\geq j$ for $j=1,\dots,r-k$.
\end{definition} 

Remark that both items~\ref{it8} and ~\ref{it9} of Theorem~\ref{th:q_list} are in this subfamily $EB_r$. The following explicit product formula proves and generalizes them.

\begin{proposition}
\label{prop:product_formula}
For any $c\in EB_r$ with $k$ as in Definition~\ref{defi:ebr}, we have
\begin{equation}
\label{eq:product_formula}
A_{c}(q)=q^{d_c}\qbin{r}{k}\;\prod_{i=1}^k \qint{i}^{c_i}\; \prod_{i=1}^{r-k} \qint{i}^{c_{r+1-i}},
\end{equation}
where $d_c=\sum_{j=1}^{r-k}\sum_{i=1}^j(c_{r+1-i}-1)$.
\end{proposition}

Note that by~\eqref{eq:product_formula}, $d_c$ is the smallest exponent of $q$ occurring in $A_c(q)$. We give a formula for this exponent in~\eqref{eq:valuation} that is valid for any $c\in\wc{r}{}$, and specializes to the expression above when $c\in EB_r$.

\begin{proof}
We use the probabilistic interpretation of $A_c(q)$ in Section~\ref{sub:proba_process} in its sequential version. We choose the word 
\[\mathbf{i}=1^{c_1}2^{c_2}\cdots k^{c_k} \;r^{c_r} (r-1)^{c_{r-1}}\cdots (k+1)^{c_{k+1}}.\]
It corresponds to the following  dropping order of particles: start with all $c_1$ particles at site $1$, then all $c_2$ particles at site $c_2$, and so on until we drop $c_k$ particles at site $k$. Then, drop all particles at site $r$, then at site $r-1$, and so on down to $k+1$.

The first part of the definition of $EB_r$ implies that in order to end with the stable configuration $\{1,\ldots,r\}$, the following must hold: for $i\leq k-1$, the intermediate stable configuration after the $i$th step is the interval $\{1,\ldots,i\}$, and the $i+1$-th particle either drops at site $i+1$ or drops on the previous interval and exits to the right. The probability of this happening is $\prod_{i=1}^k\qint{i}^{c_i}/\qfact{k}$ as is readily computed.
 
  For the remaining particles, the situation is symmetric of the first part, since the stable interval $\{1,\ldots,k\}$ does not interfere with the analysis. The probability of success in this second half is worked out to be $q^{d_c}\prod_{i=1}^{r-k} \qint{i}^{c_{r+1-i}}/\qfact{r-k}$.
 
 Now by \eqref{eq:Ac_proba}, we obtain $A_c(q)$ by multiplying these two expressions together with $\qfact{r}$, thus giving the desired expression.
 \end{proof}
 
In terms of the tree interpretation from Section~\ref{sec:postnikov_trees}, we have in fact shown that for words $\mathbf{i}$ considered in the preceding proof, there is a unique compatible tree.
We further note that Proposition~\ref{prop:product_formula}
is the $q$-analogue of~\cite[Theorem 4.4]{Liu16}. Finally, the expression for $A_c(q)$ for $c\in EB_r$ motivates understanding the valuation $d_c$ in general; we return to this in Section~\ref{sec:degree}.

\begin{example}
Consider $c=(2,0,1,0,2,1)\in EB_6$. Then the $k$ in Definition~\ref{defi:ebr} equals $3$. 
We have $d_c=c_6-1+c_6+c_5-2+c_6+c_5+c_4-3=1$. The reader may verify that
\[
A_c=q^{1}\qbin{6}{3}\qint{1}^2\qint{3}\,\qint{2}^2\qint{1}.
\]
\end{example}

\subsection{Interval support and $q$-hit numbers}
We now focus on the case where\emph{ $\supp(c)$ is an interval}. That is, $c$ can be written as $c=0^i\beta 0^{r-k-i}$ where $\beta\vDash r$ is a (strong) composition with $k\coloneqq \ell(\beta)$. 

By~\cite[Proposition 5.6]{NT21} we have
\begin{equation}
\label{eq:gf_connected}
\sum_{j\geq 0}t^j\prod_{i=1}^k\qint{j+i}^{\beta_i}=\frac{\sum_{i=0}^{r-k}A_{0^i\beta 0^{r-k-i}}(q) t^i}{(t;q)_{r+1}}.
\end{equation}

\begin{remark}If $\beta$ has a single part, so $\beta=(r)$, we get 
\begin{equation}
\label{eq:q_Eulerian}
\sum_{j\geq 0}\qint{j+1}^{r}t^j=\frac{\sum_{i=0}^{r-k}A_{0^i,r, 0^{r-k-i}}(q) t^i}{(t;q)_{r+1}}.
\end{equation}
This shows Theorem~\ref{th:q_list}\ref{it3}, since the left-hand side was already considered by Carlitz~\cite{Carlitz} and this shows by comparison that $A_{0^i,r, 0^{r-k-i}}(q)$ counts permutations in $\sgrp_r$ with $i$ descents with $q$-weight given by the major index. An alternative proof of this is given in~\cite{Mit20}.
\end{remark}

As briefly touched upon in \cite[Section 5.3]{NT21}, the family of remixed Eulerian numbers $A_{0^i\beta 0^{r-k-i}}(q)$ coincides with polynomials enumerating \emph{$q$-hit numbers} appearing in the  work of Garsia-Remmel~\cite{GarsiaRemmel}. This observation is also instrumental in order to relate these numbers to recent work around chromatic symmetric functions, see~\cite{NTchromatic}. We now justify this claim. 


Fix $r$ and consider $\lambda=(\lambda_1,\lambda_2,\ldots,\lambda_r)$ with $\lambda_i$ integers satisfying $r\geq \lambda_1\geq\cdots\geq\lambda_r\geq 0$. 
So $\lambda$ can be seen as a Young diagram in an $r\times r$ square, see Figure~\ref{fig:hits} with $r=6$ and $\lambda=(5,5,3,3,3,0)$.  Following \cite{GarsiaRemmel} the $q$-hit numbers $H_j(\lambda,q)$ can be defined by: 
\begin{align}
\label{eq:qhit_gf}
\sum_{j\geq 0}t^j\prod_{i=1}^{r}\qint{i-\lambda_{r+1-i}+j}=\frac{\sum_{j=0}^r H_j(\lambda,q)t^j}{(t;q)_{r+1}}.
\end{align}

\begin{remark}For $q=1$, the hit numbers $H_j(\lambda)\coloneqq H_j(\lambda,1)$ enumerate permutations in $\sgrp_{r}$ whose associated graph has exactly $j$ points inside of the shape $\lambda$. Figure~\ref{fig:hits} (right) shows a permutation contributing to $H_2(\lambda)$.

There exists a refinement of this interpretation that gives the $q$-hit numbers \cite{Dworkin}. We refrain from stating it given its technicality and particularly as we do not need it in the sequel, and simply note that Theorem~\ref{th:q_list}\ref{it6} can be deduced from this combinatorial interpretation.
\end{remark}

\begin{figure}
\begin{center}
\includegraphics[scale=0.4]{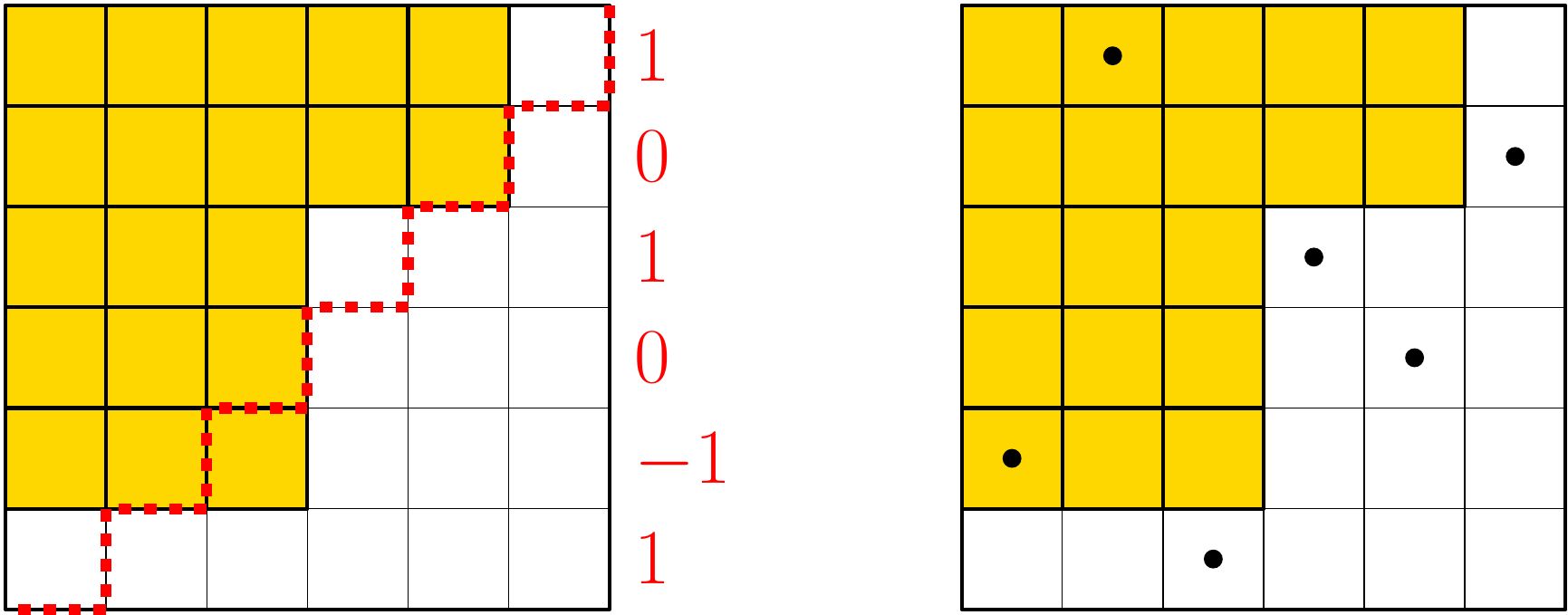}
\end{center}
\caption{\label{fig:hits} Partition in a square (left), superimposition with the graph of a permutation.}
\end{figure}

We will now compare the left-hand sides of \eqref{eq:gf_connected} and \eqref{eq:qhit_gf}. We will show that every nonzero $q$-hit polynomial is a remixed Eulerian number $A_c(q)$ where $c$ has interval support, and vice versa.

 For $i=1,2,\dots, r$, the values $i-\lambda_{r+1-i}$ occurring in~\eqref{eq:qhit_gf} go from $1-\lambda_r\leq 1$  to $r-\lambda_1\geq 0$ with successive differences in $\{1,0,-1,-2,\dots\}$. 
Geometrically these values express the algebraic distance in each row between $\lambda$ and the staircase $\delta^r\coloneqq(r,r-1,\ldots,1)$, computed from bottom to top. In our running example, we have the sequence of values $1,-1,0,1,0,1$, as Figure~\ref{fig:hits} reveals. 
 
A moment's thought shows that these values must thus form an interval, say $[a,a+k-1]$ with $k\geq 1$, that necessarily contains $0$ or $1$. Define $\beta_i(=\beta_i(\lambda))$ to be number of times that the values  $a+i-1$ is obtained, so that $B(\lambda)\coloneqq(\beta_1,\ldots,\beta_k)$ is a composition of length $k$. The coefficient $C_j$ of $t^j$ on the left-hand side of  \eqref{eq:qhit_gf} can be rewritten as
\[\prod_{i=1}^k\qint{j+a+i-1}^{\beta_i}.\]
 
 If $a=1$, we recognize this as the coefficient of $t^j$ on the left-hand side of \eqref{eq:gf_connected} with $\beta=B(\lambda)$. We thus have $H_j(\lambda,q)=A_{0^j\beta 0^{r-k-j}}(q)$ for $j=0,\ldots,r-k$ and is zero otherwise.
 
 If $a\neq 1$, one has $a\leq 0$ to ensure that $0$ or $1$ belong to $[a,a+k-1]$, which also implies $k\geq 1-a$. Thus $C_j=0$ for $j=0,\ldots,-a$ and $C_{j-a+1}=  \prod_{i=1}^k\qint{j+i}^{\beta_i}$. The left-hand side of \eqref{eq:qhit_gf} is then $t^{1-a}$ times the left-hand side of \eqref{eq:gf_connected}. It follows that $H_j(\lambda,q)=A_{0^{j-1+a}\beta 0^{r-k+1-a-j}}(q)$ for $j=1-a,\ldots,r-k+1-a$ and is zero otherwise.\smallskip

We have shown that every nonzero $q$-hit polynomial is a remixed Eulerian number $A_c(q)$ where $c$ has interval support. Conversely, given a composition $\beta=(\beta_1,\ldots,\beta_k)$, it is always possible to construct $\lambda$ such that $B(\lambda)=\beta$, and thus $H_j(\lambda,q)=A_{0^j\beta 0^{r-k-j}}(q)$ as above.

Indeed,  let $l^r(\beta)=(l_1,\ldots,l_r)$ be obtained by concatenating $\beta_k$ occurrences of $k$, $\beta_{k-1}$ occurrences of $k-1$, and so on until $\beta_1$ occurrences of $1$. If one defines $\lambda\coloneqq \delta_r-l^r(\beta)$ with componentwise subtraction, then $\lambda$ corresponds to a Ferrers shape inside an $r\times r$ square that satisfies $B(\lambda)=\beta$.

\section{Degree, symmetry, unimodality}
\label{sec:degree}

In this section we study the $A_{c}(q)$ as polynomials in $q$. They are known to have nonnegative integer coefficients, cf.~\cite[Proposition 5.4]{NT21}. We will revisit the proof of this fact to show that the $A_{c}(q)$ are \emph{symmetric} and \emph{unimodal}. To do this, we first need to determine the degree and valuation of $A_{c}(q)$.

\subsection{Degree and valuation}

The following simple relation was recorded as Theorem~\ref{th:q_list}\ref{it2} in the introduction. For any finite sequence $a=(a_1,a_2,\ldots,a_k)$, define $\rev{a}=(a_k,a_{k-1},\ldots,a_1)$.

\begin{lemma}
\label{lem:reversal symmetry}
Let $c=(c_1,\dots,c_{r})\in \wc{r}{}$. Then we have
\[
A_c(q)=q^{\binom{r}{2}}A_{\rev{c}}(q^{-1}).
\]
\end{lemma}
\begin{proof}
This is straightforward from the probabilistic perspective. Indeed we can reverse the lattice of sites $\bZ$, and then consider the IDLA process obtained by swapping $q_L$ and $q_R$. The probability of reaching the stable configuration $[r]$ starting from $\rev{c}$ in this flipped process (with $q=q_R/q_L$) is the same as that of reaching $[r]$ starting from $c$ in the original description (with $q^{-1}=q_L/q_R$). This implies the equality in question.
\end{proof}

Given $c\in\wc{r}{}$, we let $D_c$ denote the \emph{degree} of $A_c(q)$. We also  define $d_c$ as the \emph{valuation} of $A_c(q)$, i.e. the smallest exponent of $q$ with a nonzero coefficient.

\begin{remark}
\label{remark:constant_term}
Before the general case, let us characterize the $c$ such that $d_c=0$. Equivalently, we need to determine when the constant term $A_c(0)$ vanishes. From the probabilistic process point of view, $q=0$ corresponds to particles only jumping right. It is then easily shown that $A_c(0)$ is zero unless $c$ satisfies $\sum_{i\leq j} c_i\geq j$ for $j=1,\dots,r$, in which case $A_c(0)=1$. It follows immediately that $d_c=0$ if and only if that condition is satisfied.
\end{remark}

It turns out that  $A_c(q)$ is symmetric with respect to the interval $[d_c,D_c]$, as we shall show in Theorem~\ref{th:Actilde_symmetric}. To this end, we collect some notation that helps describe $D_c$ and $d_c$ combinatorially. For $t\in\bR$, we use
\[
t^+\coloneqq \max(0,t).
\]
The following pictorial perspective for $c\in \wc{r}{}$ is occasionally useful. Attach a \L ukasiewicz path $P_c$ with $c=(c_1,\dots,c_r)$ by starting at the origin and translating by $(1,c_i-1)$ as $c$ is read from left to right.
See Figure~\ref{fig:path_to_compute_Sc}.
For $i\in [r]$, we define $h_i(c)$ to be the ordinate on $P_c$ after the $i$th step. More precisely, $h_i(c)=\sum_{1\leq j\leq i}(c_j-1)$.
Note that $h_r(c)$ is necessarily $0$.
We let 
\begin{align}
H(c)&\coloneqq \sum_{1\leq i\leq r-1}h_i(c),\\
H^{-}(c)&\coloneqq\displaystyle \sum_{1\leq i\leq r-1}(-h_i(c))^{+}.
\end{align}

\begin{figure}[ht]
  \begin{tikzpicture}[scale=.7]
    \draw[gray,very thin] (0,0) grid (10,4);
    \draw[line width=0.25mm, black, <->] (0,2)--(10,2);
    \draw[line width=0.25mm, black, <->] (1,0)--(1,4);
    \node[draw, circle,minimum size=5pt,inner sep=0pt, outer sep=0pt, fill=blue] at (1, 2)   (b) {};
    \node[draw, circle,minimum size=5pt,inner sep=0pt, outer sep=0pt, fill=red] at (2, 1)   (c) {};
    \node[draw, circle,minimum size=5pt,inner sep=0pt, outer sep=0pt, fill=blue] at (3, 3)   (d) {};
    \node[draw, circle,minimum size=5pt,inner sep=0pt, outer sep=0pt, fill=blue] at (4, 2)   (e) {};
    \node[draw, circle,minimum size=5pt,inner sep=0pt, outer sep=0pt, fill=red] at (5, 1)   (f) {};
    \node[draw, circle,minimum size=5pt,inner sep=0pt, outer sep=0pt, fill=red] at (6, 0)   (g) {};
    \node[draw, circle,minimum size=5pt,inner sep=0pt, outer sep=0pt, fill=red] at (7, 0)   (h) {};
    \node[draw, circle,minimum size=5pt,inner sep=0pt, outer sep=0pt, fill=blue] at (8, 2)   (i) {};
    \draw[blue, line width=0.7mm] (b)--(c);
    \draw[blue, line width=0.7mm] (c)--(d);
    \draw[blue, line width=0.7mm] (d)--(e);
    \draw[blue, line width=0.7mm] (e)--(f);
    \draw[blue, line width=0.7mm] (f)--(g);
    \draw[blue, line width=0.7mm] (g)--(h);
    \draw[blue, line width=0.7mm] (h)--(i);
    \draw[line width=0.55mm, red, dotted] (3,2)--(3,3);
    \draw[line width=0.55mm, orange, dotted] (5,2)--(5,1);
    \draw[line width=0.55mm, orange, dotted] (6,2)--(6,0);
    \draw[line width=0.55mm, orange, dotted] (7,2)--(7,0);
    \draw[line width=0.55mm, orange, dotted] (2,2)--(2,1);
  \end{tikzpicture}
  \caption{$P_{c}$ when ${c}=(0,3,0,0,0,1,3)$.}
  \label{fig:path_to_compute_Sc}
\end{figure}
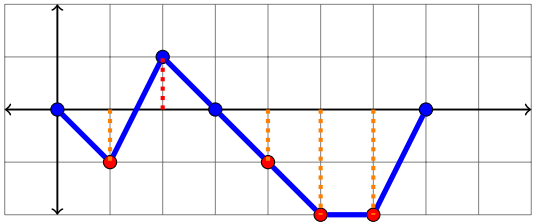

\begin{theorem}\label{th:properties_ac}
  Let $c=(c_1,\dots,c_{r})\in \wc{r}{}$.
  The following equalities hold:
  \begin{align}
  \label{eq:valuation}
  d_c&=H^{-}(c),\\
\label{eq:degree}
D_c&=\binom{r}{2}-\sum_{1\leq i\leq r-1} (h_i(c))^+.
\end{align}
\end{theorem}

The reader will note that \eqref{eq:valuation} generalizes the result of Remark~\ref{remark:constant_term}, as well as the power of $q$ in Proposition~\ref{prop:product_formula}.

\begin{proof}
First we establish \eqref{eq:valuation}.
The claim is true when $c=(1^{r})$. In this case $A_c(q)=\qfact{r}$ and we have $d_c=0=H^{-}(c)$.
We proceed by (downward) induction on the number of parts in $c$ that equal $1$.

Consider $i\in [r]$  such that $c_i\geq 2$.
  Let $[a,b]$ be the maximal interval in $\supp(c)$ containing $i$.  We have by \eqref{eq:recurrence_A} that 
   \begin{equation}
\label{eq:recurrence_A_again}
\qint{b-a+2}A_{c}(q)=q^{i-a+1}\qint{b-i+1}A_{L_i(c)}(q)+\qint{i-a+1}A_{R_i(c)}(q).
\end{equation}
It follows that the following relation holds
\begin{equation}
\label{eq:recurrence_valuation}
d_c=\min(i-a+1+d_{L_i(c)}, d_{R_i(c)}).
\end{equation}
  We want to check that $H^-(c)$ satisfies the same recurrence.
 Consider the boundary cases first. If $a=1$, only $R_i(c)$ is defined, and it is easily checked that $H^{-}(c)=H^{-}(R_i(c))$.
 If $b=r$, only $L_i(c)$ is defined and we need to check that $H^{-}(c)=i-a+1+H^{-}(L_i(c))$.
 Note that we have $c_{j}>0$ for all $j\in [a,r]$ and $c_{a-1}=0$. From this, and the fact that  $h_{r}(c)=0$, it follows that $h_j(c)\leq 0$ for all $j\in [a-1,r]$.
 Furthermore since $c_i\geq 2$ we are in fact guaranteed that $h_{j}(c)<0$ for $j\in [a-1,i-1]$.
 Now $h(L_i(c))$ is obtained from $h(c)$ by adding $1$ to all entries $h_j(c)$ for $j\in [a-1,i-1]$ and leaving the rest unchanged. Since $h_j(c)<0$ for such $j$, we still have $h_j(L_i(c))\leq 0$. It follows that $H^{-}(c)=i-a+1+H^{-}(L_i(c))$ in this case.

 We now consider the generic situation where $1<a\leq b<r$.
 The arguments are similar to those just presented, and we keep the exposition terse.
 We set $h_0(c)=0$.
Now note that the hypotheses on $a,b,i$ are equivalent to \[h_{a-2}>h_{a-1}\leq h_{a}\leq \dots \leq h_{b-1}\leq h_b>h_{b+1}\]
and $h_{i-1}<h_{i}$. Then $h(L_i(c))$ is obtained by adding from $h(c)$ by adding $1$ to $h_{a-1},\ldots,h_{i-1}$, while $h(R_i(c))$ is obtained from $h(c)$ by subtracting $1$ from $h_{i},\ldots,h_{b}$.

It is clear from this description that $H^-(c) \leq i-a+1+H^-({L_i(c)})$ and $H^-(c)\leq H^-({R_i(c)})$.
To show that it is equal to one of them, consider the sign of $h_i$. If $h_i>0$, then $h_i,\ldots,h_b>0$ from which it follows $H^-(c)=H^-(R_i(c))$. If $h_i\leq 0$, then $h_{a-1},\ldots,h_{i-1}< 0$, and so these $i-a+1$ values imply $H^-(c)=i-a+1+H^-({L_i(c)})$.

We finally establish~\eqref{eq:degree}.
It follows from Lemma~\ref{lem:reversal symmetry} that
\[
d_{\rev{c}}+D_c=\binom{r}{2}.
\]
Now $h_i(\rev{c})=-h_{r-i}(c)$ for $i\in [0,r]$ by a direct computation. Using~\eqref{eq:valuation}, the claim follows.
\end{proof}

In terms of the path $P_c$, we have that $d_c$ equals the sum of the absolute values of the heights of all lattice points that lie strictly below the $x$-axis, and $D_c$ equals $\binom{r}{2}-$sum of the heights of all lattice points that are strictly above the $x$-axis. 

\begin{example}
\label{ex:A_c example}
For $c=(0,3,0,0,0,1,3)$, the \L ukasiewicz path in Figure~\ref{fig:path_to_compute_Sc} tells us that
\begin{align*}
d_c&=1+1+2+2=6\\
D_c&=\binom{7}{2}-1=20.
\end{align*}
As a matter of fact, the full polynomial $A_{c}(q)$  is given by
\[
2 q^{20} + 6 q^{19} + 11 q^{18} + 18 q^{17} + 27 q^{16} + 35 q^{15} + 40 q^{14} + 42 q^{13} + 40 q^{12} + 35 q^{11} + 27 q^{10} + 18 q^{9} + 11 q^{8} + 6 q^{7} + 2 q^{6}.
\]
\end{example}

We remark that we have a symmetric polynomial above. 
This is quite surprising, as no symmetry is apparent in $c$. It turns out to be a general fact, valid for all polynomials $A_c(q)$, that we will now prove together with unimodality.

\subsection{Symmetry and unimodality}
\label{sub:symmetry_unimodality}

Let us say that a polynomial $P(q)$ is $\psu(N)$ if it has positive coefficients and is unimodal and symmetric with respect to $N/2$. Note that $N$ is then the sum of the degree and valuation of $P$. We have the following classical properties, see~\cite{Haglund} for instance:
\begin{lemma}
\label{lem:psu}
If $P$ and $Q$ are $\psu(N)$, then $P+Q$ is $\psu(N)$.\\
If $P$ is $\psu(N)$ and $Q$ is $\psu(N')$, then $PQ$ is $\psu(N+N')$.
\end{lemma}

We introduce the \emph{reduced} remixed Eulerian numbers \cite{NT21}
\begin{equation}
\label{eq:Actilde}
\tilde{A}_{c}(q)\coloneqq \frac{A_{c}(q)}{\prod_j \qfact{m_j}}
\end{equation}
where $m_1,\dots,m_p$ are the cardinalities of the maximal intervals $I_1,\dots,I_p$ in $\supp(c)$, ordered from left to right. Using the notations $i,a,b,L_i(c),R_i(c)$ in the recurrence relations~\eqref{eq:recurrence_A} for ${A}_{c}(q)$, the $\tilde{A}_{c}(q)$ satisfy the modified recurrence relation (\cite[Proof of Prop. 5.4]{NT21})
\begin{equation}
\label{eq:recurrence_Areduced}
\tilde{A}_{c}(q)=b_L\;q^{i-a+1}\qint{b-i+1}\;\tilde{A}_{L_i(c)}(q)+b_R\;\qint{i-a+1}\;\tilde{A}_{R_i(c)}(q),
\end{equation} 

where $b_L,b_R$ are defined as follows: if $j$ is the index such that $I_j=[a,b]$, then
\begin{itemize}
\item if $j>1$ and $I_{j-1}=[f,a-2]$ for a certain $f$, then $b_L=\qbin{b-f+1}{b-a+2}$. Otherwise, $b_L=1$.
\item if $j<p$ and $I_{j+1}=[b+2,g]$ for some $g\leq r$, then  $b_R=\qbin{g-a+1}{b-a+2}$. Otherwise, $b_R=1$.
\end{itemize}

In any case, note that both $b_R$ and $b_L$ are $\psu(N)$ for some $N$ (in general distinct):  indeed the Gaussian binomial coefficient $\qbin{n}{k}$ is known to be $\psu(k(n-k))$.

\begin{theorem}
\label{th:Actilde_symmetric}
For any $c\in \wc{r}{}$, we have that $\tilde{A}_c(q)$ is $\psu\left(\binom{r}{2}-H(c)-\sum_j\binom{m_j}{2}\right)$. It follows that $A_c(q)$ is $\psu\left(\binom{r}{2}-H(c)\right)$. 
\end{theorem}
\begin{proof}
We establish some notation for convenience.
Suppose the lengths of the maximal intervals in $\supp(c)$ are $m_1$ through $m_p$. We let 
\begin{align}
n(c)\coloneqq \sum_{1\leq j\leq p}\binom{m_j}{2}.
\end{align}

We have $A_{c}=\tilde{A}_{c}\prod_j \qfact{m_j}$, and each $\qfact{m_j}$ is $\psu(\binom{m_j}{2})$. 
By Lemma~\ref{lem:psu}, the result for $A_c$ thus follows from the one for $\tilde{A}_{c}$.  We proceed to prove the latter by induction on $e(c)\coloneqq |c|-|\supp(c)|$.

The base case corresponds to $c=(1^r)$, for which $\tilde{A}_{c}(q)=1$ and the claim is clearly true.
Now assume $e(c)>0$, and consider the recurrence~\eqref{eq:recurrence_Areduced} for a fixed $i$ with $c_i\geq 2$: we retain the notations $a,b,L_i(c),R_i(c)$.
Note first that $d_c+D_c=\binom{r}{2}-H(c)$ by summing the expressions in~\eqref{eq:valuation} and~\eqref{eq:degree}. 
It then follows from the description of $h(L_i(c))$ and $h(R_i(c))$ given in the proof of~\eqref{eq:valuation} that $d_{L_i(c)}+D_{L_i(c)}=d_c+D_c-(i-a+1)$ and $d_{R_i(c)}+D_{R_i(c)}=d_c+D_c+(b-i+1)$.

Since $e(L_i(c))=e(R_i(c))=e(c)-1$,  we can apply induction to conclude that $\tilde{A}_{L_i(c)}(q)$ and $\tilde{A}_{R_i(c)}(q)$  are $\psu\left(\binom{r}{2}-H(L_i(c))-n(L_i(c))\right)$ and $\psu\left(\binom{r}{2}-H(R_i(c))-n(R_i(c))\right)$ respectively. 
Let $B_i$ and $C_i$ denote the left and right summands on the right-hand side of ~\eqref{eq:recurrence_Areduced}. 
It will suffice to show that $B_i$ and $C_i$ are both $\psu\left(\binom{r}{2}-H(c)-n(c)\right)$; by Lemma~\ref{lem:psu} so is their sum $A_c$ which then completes the proof.

We focus on $B_i$, the proof for $C_i$ being entirely similar. 
Assume first that $b_L=1$. Then 
\begin{align}
\label{eq:intermediate psu b_i}
\deg(B_i)+\val(B_i)&=2(i-a+1)+b-i+\binom{r}{2}-H(L_i(c))-n(L_i(c))\nonumber\\
&=b+i-2a+2+\binom{r}{2}-H(L_i(c))-n(L_i(c)).
\end{align}
Note that $H(L_i(c))-H(c)=i-a+1$.
Since $b_L=1$, we know that $n(L_i(c))-n(c)=b-a+1$.
Thus we may rewrite~\eqref{eq:intermediate psu b_i} as
\begin{align}
\label{eq:desired}
\deg(B_i)+\val(B_i)=\binom{r}{2}-H(c)-n(c),
\end{align}
which disposes off the case $b_L=1$.

To finish this proof, we finally consider the case where $b_L=\qbin{b-f+1}{b-a+2}$.
This time we get
\begin{align}
\deg(B_i)+\val(B_i)&=b+i-2a+2+(a-f-1)(b-a+2)+\binom{r}{2}-H(L_i(c))-n(L_i(c))\nonumber\\
&=i-a+(a-f)(b-a+2)+\binom{r}{2}-H(L_i(c))-n(L_i(c)). \label{eq:intermediate psu b_i take 2}
\end{align}
Like before, the equality $H(L_i(c))-H(c)=i-a+1$ holds. Unlike before, we have
\begin{align}
n(L_i(c))-n(c)&=\binom{b-f+1}{2}-\binom{a-f-1}{2}-\binom{b-a+1}{2}\\\nonumber 
&=(a-f)(b-a+2)-1.
\end{align}
We leave it to the reader to put the pieces together and conclude that~\eqref{eq:desired} holds.
\end{proof}

\begin{example}
Recall that we computed $A_c(q)$ for $c=(0,3,0,0,0,1,3)$  in Example~\ref{ex:A_c example}, and it is seen to be $\psu(26)$, which is in accordance with Theorem~\ref{th:Actilde_symmetric}.  Indeed, $\binom{7}{2}-H(c)=21+5=26$.
\end{example}

\section{$q$-volumes and a dissection of the permutahedron}
\label{sec:volumes}

In this section and the next we will give a second combinatorial interpretation of $A_c(q)$ after the one in Section~\ref{sec:postnikov_trees}. As we will show in Section~\ref{sec:A_c bilabeled}, it can be interpreted as extending the one given by Liu~\cite{Liu16} for $q=1$. In order to interpret the parameter $q$, we will use a decomposition of the permutahedron into cubes   which is of independent interest.

\subsection{The polytopes $\cube_\lambda(u)$}
\label{sub:cubes}

Fix $\lambda\in \mathbb{R}^{r+1}=(\lambda_1,\lambda_2,\dots,\lambda_{r+1})$ such that $\lambda_i\geq \lambda_{i+1}$ for $i=1,\ldots,r$. For the next definition, we embed $\sgrp_{r}$ into $\sgrp_{r+1}$ as usual by treating $r+1$ as a fixed point. 

\begin{definition}[$I_u$ and $\cube_\lambda(u)$] For $u\in\sgrp_{r}$, consider the interval $I_u=[u,s_1s_2\cdots s_{r}u]$ in the Bruhat order on $\sgrp_{r+1}$. The polytope $\cube_\lambda(u)$ is defined as the convex hull in $\bR^{r+1}$  of the points $\lambda_v$ for $v\in I_u$.
\end{definition}

We multiply permutations from right to left, so if $u=u_1,\ldots,u_r,r+1$ in one-line notation, then $s_1s_2\cdots s_{r}u=u_1+1,\ldots,u_r+1,1$.

\begin{example}
For $r=1$, we have $u=1$ which $I_1=[12,21]=\{12,21\}\subseteq\sgrp_2$.\\
 For $r=2$,  we get $I_{12}=[123,231]=\{123,213,132,231\}$ and $I_{12}=[213,321]=\{213,231,312,321\}$. \\
For $r=3$ and $\lambda=(3,2,1,0)$ the six polytopes $\cube_\lambda(u)$ for $u\in\sgrp_r$ are illustrated at the bottom left of Figure~\ref{fig:3slicing}. 
They are all \emph{Bruhat interval polytopes} introduced by Tsukerman--Williams~\cite{TW15}.
\end{example}

\begin{figure}
\includegraphics[scale=0.5]{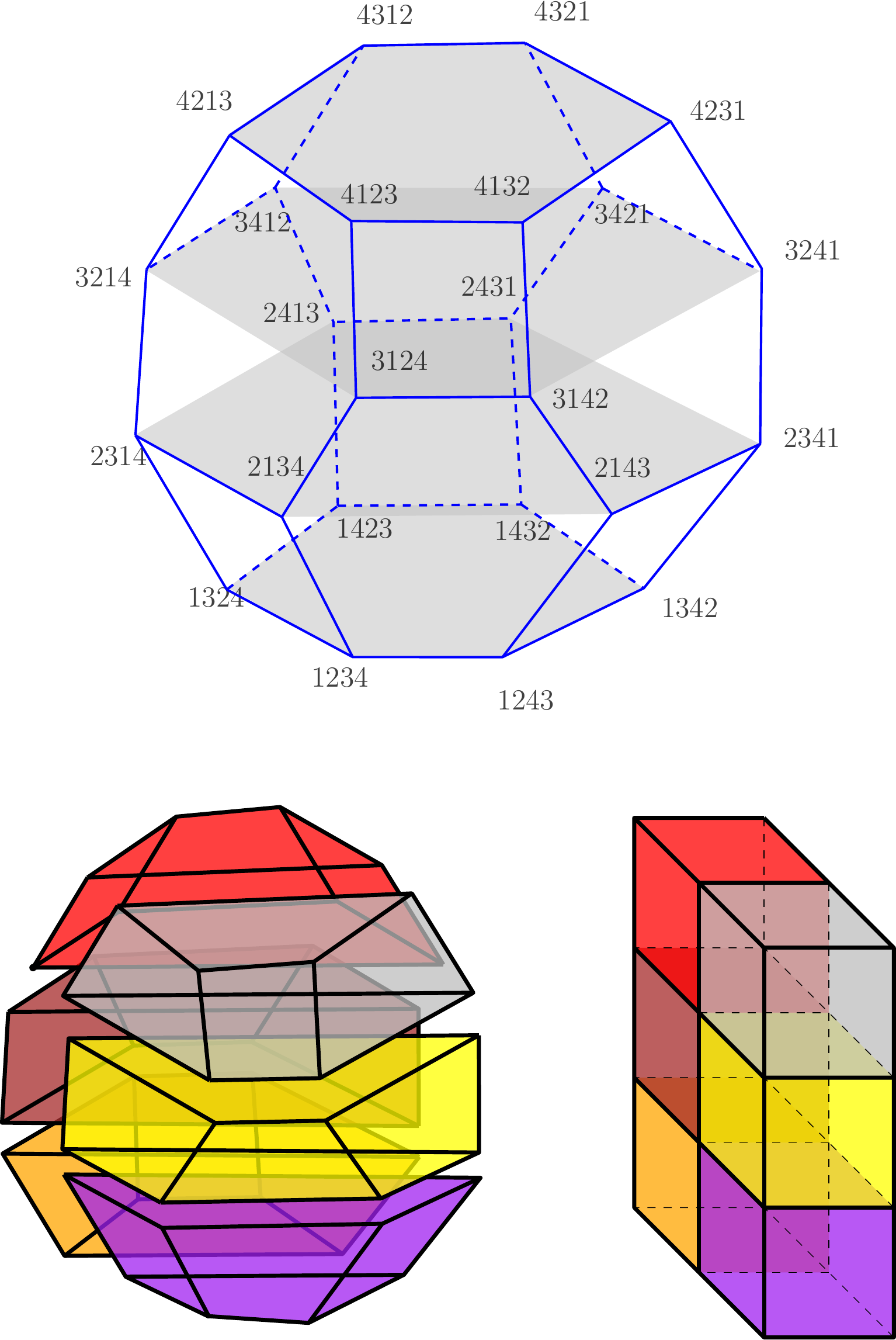}
\caption{Slicing of the three dimensional permutahedron (top), its full dissection into cubes (bottom left)  and the associated cubical complex (bottom right).\label{fig:3slicing}}
\end{figure}

\begin{remark}
Recall that the $\lambda_v$ for $v\in\sgrp_{r+1}$ are the vertices of the permutahedron $\Perm(\lambda)$. It follows that the $\lambda_v$ for $v\in I_u$ are the vertices of $\cube_\lambda(u)$ for any $u\in\sgrp_r$.
\end{remark}

This description of $\cube_{\lambda}(u)$, while possessing the virtue of brevity, is not entirely convenient in practice as it requires listing elements of $I_u$. 
We now proceed to give another perspective on the intervals $I_u$ which is more enlightening. 

Recall that the \emph{code} $\code(w)$ of a permutation $w\in \sgrp_{r+1}$ is the weak composition $(c_1,\dots,c_{r+1})$ where $c_i=|\{j>i\suchthat w_i>w_j\}|$. For instance $\code(23514)=(1,1,2,0,0)$. It sends $\sgrp_{r+1}$ bijectively to the sequences $(c_1,\dots,c_{r+1})$ such that $0\leq c_i\leq r+1-i$ for all $i$.

\begin{lemma}
For any $u\in\sgrp_{r}$, the Bruhat interval $I_u$ of $\sgrp_{r+1}$ is isomorphic to the Boolean lattice $\mathbb{B}_{r}$ of cardinality $2^r$.
\end{lemma}
\begin{proof}

Denote the bottom and top element of $I_u$  respectively by $u_-=u_1,\dots,u_r,r+1$ and $u_+=u_1+1,\dots,u_r+1,1$ in one-line notation. 

Given $S\subset [r]$, let $\mbf{e}_S\in \bR^{r+1}$ denote the indicator vector of $S$, i.e. $\mbf{e}_S$ equals the sum of the standard basis vectors $e_i$ (in $\bR^{r+1}$)  for $i\in S$. Define
\[
J_u \coloneqq\{c=(c_1,\ldots,c_{r+1})\suchthat c=\code(u_-)+\mbf{e}_S \text{ for } S\subset [r]\}. 
\]

We claim that $v\in I_u$  (i.e $u_-\leq v\leq u_+$ in Bruhat order) if and only $\code(v)\in J_u$, and that this gives a poset isomorphism (the order on weak compositions is componentwise). Note that this proves the lemma since $J_u$ is clearly order isomorphic to $\mathbb{B}_r$. 

We have $\code(u_+)=\code(u_-)+\mbf{e}_{[r]}$ as is checked immediately. Therefore $J_u$ consists of all weak compositions $c$ such that $\code(u_-)\leq c \leq \code(u_+)$.

We leave it to the reader to check that the full claim follows from invoking the following well-known \emph{tableau criterion}: $u\leq v$ in Bruhat order if for any $i$, the weakly increasing rearrangement of $u_1\ldots u_i$ is  smaller, componentwise, than the weakly increasing rearrangement of $v_1\ldots v_i$. In view of this, the condition $u_-\leq v\leq u_+$ translates to exactly 2 choices for each $v_i$ as $i$ goes from $1$ through $r$.
\end{proof}

\subsection{The cube decomposition}
\label{sub:cube_decomposition}

We recall the classical bijection $u\mapsto \mathrm{T}(u)$ between $\sgrp_{r}$ and {\em decreasing binary trees}. $\mathrm{T}(u)$ is defined more generally for $u$ a word with distinct letters in $\bZ_{>0}$. 
Assume $u=u_LMu_R$ where $M$ is the maximal integer in $u$. 
Then $\mathrm{T}(u)$ is constructed recursively as the binary tree with root label $M$ whose left (resp. right) subtree is $\mathrm{T}(u_L)$ (resp. $\mathrm{T}(u_R)$). 
For instance, the tree in Figure~\ref{fig:dec_tree} corresponds to the permutation $u=47128635$. 
\begin{figure}[!ht] 
\includegraphics[scale=1.5]{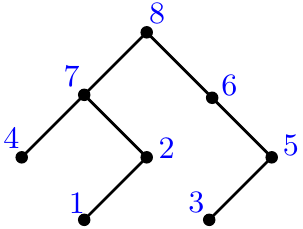}
\caption{Decreasing tree $\mathrm{T}(u)$  for $u=47128635$.}
\label{fig:dec_tree}
\end{figure}

Given $u\in \sgrp_{r}$ and $i\in [r]$, let $L(u,i)\coloneqq\{i-f_i,\dots,i-1\}$ and $R(u,i)\coloneqq\{i+1,\dots,i+g_i\}$, where $f_i,g_i\geq 0$ are maximal such that $u_j<u_i$ for $j\in L(u,i)\cup R(u,i)$. These are the labels of the descendants of the node $i$ in $\mathrm{T}(u)$: more precisely, $L(u,i)$ (resp. $R(u,i)$) is the set of labels of the left (resp. right) subtree of that node.\smallskip

{\em We now assume $\lambda_i>\lambda_{i+1}$ for $i=1,\ldots,r$}. 
Although the next theorem can be adapted for cases with identical values, we do not need it since we will be interested in the volume of $\Perm(\lambda)$, for which we will obtain polynomial formulas in the $\lambda_i$ that will still be valid in the general case.

\begin{theorem}
\label{th:H_description}
The $r!$ polytopes $\{\cube_\lambda(u)\}_{u\in\sgrp_{r}}$ are combinatorial cubes that are the maximal faces of a polyhedral subdivision of $\Perm(\lambda)$.

The facet description of $\cube_\lambda(u)$ is given by
 \begin{align}
 \label{eq:facet description1}
\emph{\text{(left)}} \hspace{5mm} t_{u_i}+\sum_{j\in L(u,i)}t_{u_j}&\leq \lambda_{i-f_i}+\dots+\lambda_i; \\
 \label{eq:facet description2}
\emph{\text{(right)}} \hspace{5mm} t_{u_i}+\sum_{j\in R(u,i)}t_{u_j}&\geq \lambda_{i+1}+\dots+\lambda_{i+g_i+1},
 \end{align}
 for $i=1,\ldots,r$.
\end{theorem} 

\begin{proof}
Recall that we assume $\lambda_i>\lambda_{i+1}$ for $i=1,\ldots,r$. We have that $\Perm(\lambda)$ is contained between the hyperplanes $x_1=\lambda_1$ and $x_1=\lambda_{r+1}$. 

Now fix $i$, $1\leq i\leq r$, and consider the slice $P_i\coloneqq \Perm(\lambda_1,\dots,\lambda_n)\cap \{\lambda_i\geq x_1\geq \lambda_{i+1}\}$. The sections $S_t\coloneqq\Perm(\lambda_1,\dots,\lambda_n)\cap\{x_1=t\}$ of $P_i$  for $\lambda_i\geq t\geq \lambda_{i+1}$ are explicitly described in Liu's work \cite[Proposition 3.7]{Liu16}:
\[S_t=\{t\}\times\Perm(\lambda_1,\dots,\lambda_{i-1},\lambda_i+\lambda_{i+1}-t,\lambda_{i+2},\ldots,\lambda_n)\]
 It follows from that result that the slice $P_i$ is in particular combinatorially equivalent to the product of $[0,1]$ with a permutahedron of dimension one less. Proceeding inductively, one can decompose into cubes the two permutahedra that appear as intersections with the hyperplanes $x_1=\lambda_i$ and $x_1=\lambda_{i+1}$. By taking the direct product with $[0,1]$, we obtain the decomposition of $\Perm(\lambda)$ into cubes.

It is then easily checked inductively that these cubes are precisely the cubes $\cube_{\lambda}(u)$, and that the facet description is as described in ~\eqref{eq:facet description1},\eqref{eq:facet description2}. 
See Figure~\ref{fig:3slicing} for an illustration of the decomposition in the case of the three dimensional standard permutahedron.
\end{proof} 

\begin{remark}
\label{remark:pitman-stanley}
When $u=id$ the $\cube_{\lambda}(u)$ is the Pitman--Stanley polytope~\cite{Stanley-Pitman}, and the latter is well known to be a combinatorial cube. As we shall see in the next section, this is directly related to setting $q=0$ in our setting. 
\end{remark}

\begin{remark}
\label{remark:generalized_permutahedron}
For any $\lambda,u$, the polytope $\cube_{\lambda}(u)$ is a \emph{generalized permutahedron} \cite{Pos09}. One needs to check that each edge of $\cube_{\lambda}(u)$ is parallel to $e_i-e_j$ for some $i,j$. Now we know that the face poset of $\cube_{\lambda}(u)$ corresponds to an interval in the Bruhat order, so adjacent vertices of $\cube_{\lambda}(u)$ have coordinates $\lambda_v$ and $\lambda_{v'}$ where $v$ covers $v'$ in Bruhat order, which implies that $\lambda_v-\lambda_{v'}$ is parallel to $e_i-e_j$ for some $i,j$ as desired.
\end{remark}
 
\subsection{$q$-volume}
We now recall the notion of $q$-volumes introduced by the authors 
\cite[\S9]{NT21}. 
Given $\lambda$, we define the $q$-volume as
\begin{align}
V^q(\lambda)=\frac{1}{\qfact{r}} \ds{(\lambda_1x_1+\cdots +\lambda_{r+1}x_{r+1})^{r}}_{r+1}^q.
\end{align}
The remixed Eulerians arise as follows:
\begin{align}
\label{eq:vol_q_remixed}
V^q(\lambda)&=
 \sum_{c\in \wc{r}{}}A_{c}(q)\;\frac{(\lambda_1-\lambda_2)^{c_1}}{c_1!}\cdots \frac{(\lambda_{r}-\lambda_{r+1})^{c_{r}}}{c_r!}.
\end{align}
At $q=1$, this recovers $V^{1}(\lambda)=\vol(\Perm(\lambda))$ \cite{Pos09}. From the subdivision of $\Perm(\lambda)$ into cubes $(\cube_{\lambda}(u))_{u\in \sgrp_r}$ it follows that 
\begin{align}
V^{1}(\lambda)=\vol(\Perm(\lambda))=\sum_{u\in \sgrp_r} \vol(\cube_{\lambda}(u)).
\end{align}
Theorem~\ref{th:q-volume-interpretation} $q$-deforms this equality by weighing  each summand by  $q^{\ell(u)}$.

As in \cite[\S6]{HarHor18}, to each $u\in \sgrp_r$ one can associate a Richardson variety 
\begin{equation}
R(u)=X_{1\times w_o^{r}u}\cap X^{u} 
\label{eq:richardson}
\end{equation}
in the variety of complete flags of $\bC^{r+1}$. 
Its cohomology class can be represented by the product of Schubert polynomials $\schub{u}\schub{1\times w_o^{r}u}$.

\begin{theorem}
\label{th:q-volume-interpretation}
For any $\lambda$, we have
\begin{equation}
\label{eq:q-volume-interpretation}
V^q(\lambda)=\sum_{u\in \sgrp_r}q^{\ell(u)} \,\vol(\cube_\lambda(u)).
\end{equation}
\end{theorem}

\begin{proof} By definition of qDS~\eqref{eq:qds}, we have
\[V^q(\lambda)=\frac{1}{\qfact{r}}\;\partial_{w_o}((\sum_i\lambda_ix_i)^rP_r)\]
with $P_r=\prod_{1\leq i<j\leq r}(qx_i-x_{j+1})$. 
Now we recognize $P_r$ as the specialization of a double Schubert polynomial and get the expansion
\[P_r=\sum_{u\in\sgrp_r}q^{\ell(u)}\schub{u}\schub{1\times w_o^{r}u},\]
using Cauchy's formula~\cite{Man01}. Then $\schub{u}\schub{1\times w_o^{r}u}$ as representing the class of the Richardson variety $R(u)$, and so the formula for $V^q(\lambda)$ above gives
\[V^q(\lambda)=\sum_uq^{\ell(u)}\mathrm{Vol}_\lambda(R(u))\]
and conclude once again with \cite[Remark 6.5]{HarHor18}: $\mathrm{Vol}_{\lambda}(R(u))$ denotes the volume polynomial attached to the variety $R(u)$ which is shown to coincide with $\vol(\cube_\lambda(u))$.
\end{proof}

Theorem~\ref{th:q-volume-interpretation} further justifies dubbing $V^q(\lambda)$ a $q$-volume. Based on this result and the expansion~\eqref{eq:volume_perm_mu_expansion}, we give a combinatorial interpretation of remixed Eulerian numbers in Section~\ref{sec:A_c bilabeled}. 
We end this section with some remarks.

\begin{remark}
Setting $q=0$ in ~\eqref{eq:q-volume-interpretation}, we have that $V^0(\lambda)$ is the volume of $\cube_\lambda(id)$. This polytope is the Pitman--Stanley polytope as noticed in Remark~\ref{remark:pitman-stanley}. Using Remark~\ref{remark:constant_term}, it is then immediate that the expansion~\eqref{eq:vol_q_remixed} at $q=0$ recovers the well-known expansion in~\cite[Theorem 1]{Stanley-Pitman}.
\end{remark}

\begin{remark}
It would be nice to have a more direct proof of the previous result, avoiding the use of the Richardson variety $R(u)$. A possibility would be to show by a direct computation that
\[v_q(\mu)=\sum_{i=1}^{r}q^{i-1}\int_{0}^{\mu_i}v_q(\mu_1,\dots,\mu_{i-2},\mu_{i-1}+\mu_i-t,t+\mu_{i+1},\mu_{i+2},\dots,\mu_r) \,\mathrm{d}t
\]
where $v_q(\mu)$ is the expression of $V^q(\lambda)$ in terms of $\mu_i=\lambda_{i}-\lambda_{i+1}$.\footnote{In personal communication, Jang Soo Kim has found a proof following this route.} Indeed this recurrence relation is satisfied by the right hand side in Theorem~\ref{th:q-volume-interpretation}, by slicing the permutahedron as in the proof of Theorem~\ref{th:H_description}. This is in fact the $q$-deformation of  \cite[Proposition 3.9]{Liu16}; we come back to this in Section~\ref{sub:liu}.
\end{remark}

\begin{remark}
Let $\mc{T}_{r}$ denote the set of (rooted) complete binary trees with $r$ internal nodes.
Fix $T\in \mc{T}_{r}$.
We denote by $\mc{S}_T$ the set of all permutations in $\sgrp_{r}$ whose decreasing tree (completed so that it has $r+1$ unlabeled leaves) has underlying shape $T$.
Theorem~\ref{th:H_description} implies that as $u$ ranges over $\mc{S}_T$, the cubes $\cube_{\lambda}(u)$ are all congruent.
Thus we may without any harm consider our cubes to be indexed by trees in $\mc{T}_{r}$. 

Endow $T$ with the binary search labeling.
This has the advantage that rewriting inequalities in Theorem~\ref{th:H_description} with the new coordinates results in inequalities of the form $t_{i}+\cdots+t_j \geq \lambda_{i+1}+\cdots + \lambda_{j+1}$ or $t_{i}+\cdots+t_j\leq \lambda_{i}+\cdots+\lambda_{j}$.
It follows from \cite[Definition 3.3]{LP20} that $\cube_{\lambda}(T)$ is an alcoved polytope.
Since $\cube_{\lambda}(T)$ is also a generalized permutahedron, we infer that $\cube_{\lambda}(T)$ is a polypositroid \cite[Definition 3.8]{LP20}. In particular, each $\cube_{\lambda}(u)$ is obtained by hitting $\cube_{\lambda}(T)$ by the permutation $u$.
\end{remark}

\begin{remark}
In view of the preceding remark, we may compactify the $V^q(\lambda)$ as a sum  over $\mc{T}_{r}$:
\begin{align}
V^q(\lambda)=\sum_{u\in \sgrp_{r}} q^{\ell(u)}\vol(\cube_{\lambda}(u))=\sum_{T\in \mc{T}_{r}} \vol(\cube_{\lambda}(T))\sum_{u\in \mc{S}_T}q^{\ell(u)}.
\end{align}
Appealing to the $q$-hooklength formula for binary trees we get
\begin{align}
\label{eq:q-vol over trees}
V^q(\lambda)=\sum_{T\in \mc{T}_{r}} \vol(\cube_{\lambda}(T))\times q^{\mathrm{stat}(T)}\frac{\qfact{r}}{\prod_{v\in T}\qfact{h_v}},
\end{align}
where $\mathrm{stat}(T)$ is an easy-to-describe\footnote{It equals the sum over all internal nodes of the number of right edges in the unique path from root to node.} statistic on $\mc{T}_{r}$ that is not relevant for the discussion at hand.
Given the expression for $q$-volume in~\eqref{eq:q-vol over trees} as a sum over binary trees involving the hooklength formula, it is natural to compare it with \cite[Theorem 17.1]{Pos09} that expresses the ordinary volume of the permutahedron as a sum over binary trees as well. When we set $q=1$ and specialize to the case of the standard permutahedron, it can be checked that $\vol(\cube_{\lambda}(T))$ does not become the product on the right-hand expression in loc. cit..
In particular, our expansion does not yield Postnikov's hooklength formula \cite[Corollary 17.3]{Pos09} despite the similarity.
\end{remark}

\section{Combinatorial interpretation via bilabeled trees.}
\label{sec:A_c bilabeled}

We turn our attention to providing a combinatorial interpretation for $A_c(q)$ using the $q$-volume perspective coupled with work of \cite{HarHor18}. 
The overarching idea is the cubes $\cube_{\lambda}(u)$ are obtained by applying as images of certain special faces of the Gelfand--Tsetlin polytope $\gz(\lambda)$ under a simple volume-preserving map.

\subsection{Gelfand--Tsetlin polytope, face diagrams, and shifted tableaux}
\label{subsec:notation for hhmp}

The \emph{Gelfand--Tsetlin} polytope $\gz(\lambda) \subset \bR^{\binom{n}{2}}$ contains all points $(x_{ij})_{1< i\leq j\leq n}$ where
\begin{align}
\label{eq:conditions GZ}
x_{i,j-1}\geq x_{i,j} \geq x_{i+1,j}
\end{align}
holds for all $1\leq i<j\leq n$. Here we assume that $x_{ii}=\lambda_i$ for all $i\in [n]$.
Points in $\gz(\lambda)$ can be interpreted as fillings of a triangular array as shown in Figure~\ref{fig:GZ general} subject to conditions in~\eqref{eq:conditions GZ}.
\begin{figure}
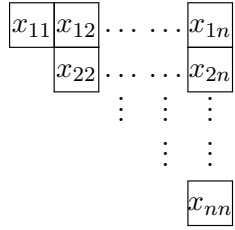

\ytableausetup{mathmode,boxsize=1.5em}
\begin{ytableau}
x_{11} & x_{12} & \none[\dots] & \none[\dots] & x_{1n}\\
\none & x_{22} & \none[\dots] & \none[\dots]  & x_{2n}\\
\none & \none & \none[\vdots]& \none[\vdots]& \none[\vdots]\\
\none & \none &  \none &\none[\vdots]& \none[\vdots]\\
\none & \none &  \none &\none & x_{nn}
\end{ytableau}
\caption{GT pattern}
\label{fig:GZ general}
\end{figure}

The authors of \cite{HarHor18} take inspiration from  work of Kogan \cite{Kog00}, Kiritchenko \cite{Kir10}, Kiritchenko--Smirnov--Timorin \cite{KST12}, and use a combinatorial gadget called \emph{face diagrams}
to emphasize equalities defining faces of $\gz(\lambda)$. 

Consider a set of $\binom{n+1}{2}$ vertices, each placed at the centre of the boxes in Figure~\ref{fig:GZ general}. We address the vertex in row $i$ from the top and column $j$ from the left by $(i,j)$.
Faces of $\gz(\lambda)$ of interest to us are determined by declaring at most one inequality among $x_{i,j} \leq x_{i,j+1}\leq x_{i+1,j+1}$ to be an equality.
Pictorially we emphasize this equality by drawing an edge between the appropriate vertices. The resulting graphs are called face diagrams.



The face diagrams we are interested in are indexed by permutations.
Pick $u\in \sgrp_{r}$ and let $(d_1,\dots,d_r)$ be such that $\code(u^{-1})=(d_1-1,\dots,d_{r}-1)$.
Consider the face $\face{u}$ of $\gz(\lambda)$ (and face diagram $\setfd(u)$)  defined by the equalities
\begin{align}
\label{eq:equalities for faces}
x_{i,i+j}=x_{i,i+j-1} \text{ for $1\leq i<d_j$ };  x_{i,i+j}=x_{i+1,i+j} \text{ for $d_j<i\leq r+1-j$.}
\end{align}
Then $\face{u}$ is $r$ dimensional.
We now observe that the association $\setfd(u)\leftrightarrow u$ is essentially the folklore bijection between decreasing binary trees and permutations described in Section~\ref{sub:cube_decomposition}.

Indeed, note the following aspect of $\setfd(u)$. 
For a fixed $j\in [r]$, there exists a unique $i\in [r+1-j]$ such that the vertex $(i,i+j)$ is  neither connected to the vertex immediately below nor connected to the vertex immediately to the left.
Indeed this vertex is given by $(d_j,d_j+j)$.
We \emph{enrich} $\setfd(u)$ by introducing edges joining $(d_j,d_j+j)$ to $(d_j+1,d_j+j)$ and $(d_j,d_j+j-1)$ for each $j\in [r]$. Additionally, we label the vertices $(d_j,d_j+j)$ by $j$.
The resulting graph is connected and has $\binom{n+1}{2}-1$ edges, and hence must be a tree (in a planar representation).
One can treat it as a rooted binary tree with root given by the vertex $(1,n)$. 
See Figure~\ref{fig:fd_to_tree} (middle) for the enriched face diagram for $u=2647351
\in\sgrp_7$.

Shrinking all paths present in the original $\setfd(u)$  to length $0$ we get a complete binary tree with $r$ labeled internal nodes and $r+1$ unlabeled leaves. Ignoring leaves, one obtains the decreasing tree $(T,\mathrm{dec})$ attached to $u$.\footnote{In the language of Section~\ref{sub:cube_decomposition} this is $\mathrm{T}(u)$.} Here $T$ records the unlabeled underlying tree and $\mathrm{dec}$  the decreasing labeling. See the decreasing tree on the right in Figure~\ref{fig:fd_to_tree}.

\begin{figure}
\includegraphics[scale=0.75]{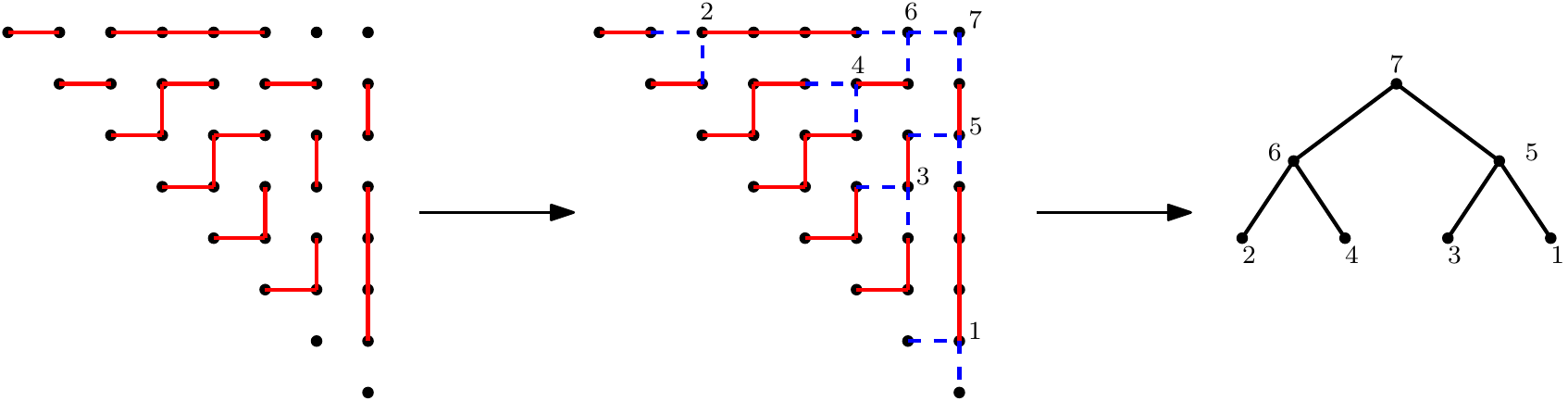}
\caption{$\mathrm{FD}(u)$ for $u=2647351
\in\sgrp_7$ (left) and the corresponding $(T,\mathrm{dec})$ (right).}
\label{fig:fd_to_tree}
\end{figure}

To give a combinatorial interpretation to $A_c(q)$, we endow a  decreasing tree (completed with unlabeled leaves) with an additional labeling $\mathrm{lr}$ of the nodes and leaves  with distinct positive integers drawn from $\{1,\dots,2r+1\}$ such that the label of any node is larger (resp. smaller) than that of its left (resp. right) child.
Furthermore, the labels on the leaves increase when read from left to right. 
We call the triple $(T,\mathrm{dec},\mathrm{lr})$ a \emph{bilabeled tree}. 
If the sequence of labels read off the leaves is $(1=\ell_1<\cdots < \ell_{r+1}=2r+1)$, we define the \emph{content} of the bilabeled tree by setting $c_i=\ell_{i+1}-\ell_i-1$ for $i\in [r]$.
 \begin{theorem}
 \label{th:remixed via trees}
 For $c\in \wc{r}{}$, let $B(c)$ be the set of bilabeled trees $(T,\mathrm{dec},\mathrm{lr})$ with content $c$.
 We have
 \[
 A_c(q)=\sum_{B(c)}q^{\ell(u)},
 \]
 where $u\in \sgrp_r$ is the permutation determined by $(T,\mathrm{dec})$.
 \end{theorem}
 \begin{proof}
It is shown in \cite{HarHor18} that, under a simple transformation, the faces $\face{u}$ of $\gz(\lambda)$ are the $\cube_{\lambda}(u)$; see~\cite[Theorem 5.4]{HarHor18} and results in Section 6 in loc. cit..
Therefore $\vol(\face{u})=\vol(\cube_{\lambda}(u))$. The former quantity can be described in terms of \emph{shifted tableaux}, which we then recast.

A \emph{shifted Young tableau} $P$ associated with $\face{u}$ is a filling of the cells in Figure~\ref{fig:GZ general} with entries from $\{1,\dots,2r+1\}$ so that they increase
weakly from left to right and top to bottom, and neighboring entries are equal if and only if the corresponding vertices are connected by an edge in $\mathrm{FD}(u
)$. The \emph{diagonal vector} of $P$ is the sequence of entries in the cells $(i,i)$ as $i$ ranges from $1$ through $r+1$.

By \cite[Proposition 3.1]{HarHor18} that
\begin{align}
\vol(\face{u})=\sum_{c\in \wc{r}{}}\sum_{P\in \mathrm{ShT}(u,c)} \frac{(\lambda_1-\lambda_2)^{c_1}}{c_1!}\cdots\frac{(\lambda_r-\lambda_{r+1})^{c_r}}{c_r!},
\end{align}
where $\mathrm{ShT}(u,c)$ is the set of 
 shifted tableaux associated to $\face{u}$ with prescribed diagonal vector $(1,c_1+2,c_1+c_2+3,\dots,c_1+\dots+c_r+r+1)$.
Theorem~\ref{th:q-volume-interpretation} then becomes
\begin{align}
V^q(\lambda)=\sum_{u\in \sgrp_r}q^{\ell(u)} \sum_{c\in \wc{r}{}}\sum_{P\in \mathrm{ShT}(u,c)} \frac{(\lambda_1-\lambda_2)^{c_1}}{c_1!}\cdots\frac{(\lambda_r-\lambda_{r+1})^{c_r}}{c_r!}.
\end{align}
Comparing with~\eqref{eq:vol_q_remixed}  yields the following interpretation:
\begin{align}
 A_c(q)=\sum_{u\in \sgrp_r}\sum_{P\in\mathrm{ShT}(u,c)} q^{\ell(u)}.
\end{align} 
Mimicking how we went from $\mathrm{FD}(u)$ to $(T,\mathrm{dec})$, we can translate shifted tableaux to bilabeled trees\textemdash{} the inequalities defining the former translate into the `local binary search' condition on the labeling $\mathrm{lr}$ while  the diagonal vector gives the leaf labeling.
\end{proof}
 
\begin{example}
Consider $c=(2,0,1)\in \wc{3}{} $. Figure~\ref{fig:bilabeled_201} shows all bilabeled trees with content $c$. The blue labels on the outside record the $\mathrm{lr}$ labeling whereas the interior labels record the decreasing labeling.
\begin{figure}
\includegraphics[scale=0.85]{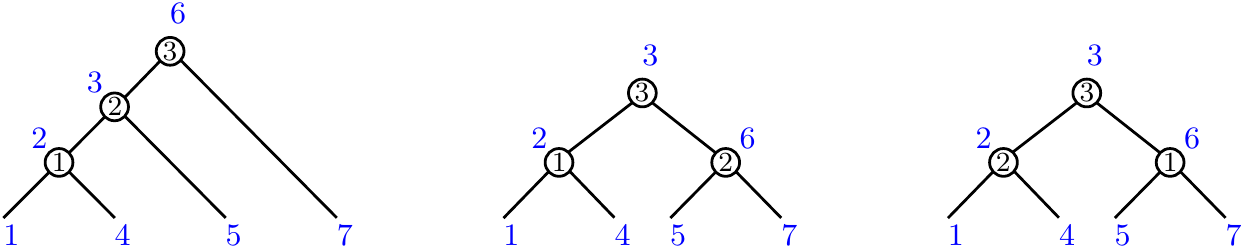}
\caption{All bilabeled trees with content $(2,0,1)$.}
\label{fig:bilabeled_201}
\end{figure}
We get
\begin{align*}
A_{201}(q)=q^{\ell(123)}+q^{\ell(132)}+q^{\ell(231)}=1+q+q^2.
\end{align*}
\end{example}

\subsection{Connection with Liu's work}
\label{sub:liu}
 
We now demonstrate how the bilabeled tree construction  is related to work of Liu \cite{Liu16}.

Given a bilabeled tree $(T,\mathrm{dec},\mathrm{lr})$, consider the node $v_1$ labeled $1$ in the decreasing labeling; its children are leaves, say with labels $\ell_i<\ell_{i+1}$ in the labeling $\mathrm{lr}$. 
Then we must have $c_i=\ell_{i+1}-\ell_i-1$.  
Let $j$ be the $\mathrm{lr}$-label of $v_1$.
The local binary search condition tells us that $\ell_i<j<\ell_{i+1}$. 

If $i=1$, then necessarily $j=2$ since $v_1$ is the only possible node with this label; similarly if $i=r$, then $j=2r$. 
We obtain a bilabeled tree $(T',\mathrm{dec}',\mathrm{lr}')$ as follows:
  \begin{itemize}
  \item $T'$ is obtained by replacing the node $v_1$ and its two attached leaves by a single new leaf $leaf$.
  \item $\mathrm{dec}'$ is obtained by decreasing $\mathrm{dec}$ by one on the remaining nodes in $T'$.
  \item $\mathrm{lr}'$ is obtained by labeling the new leaf $leaf$ by $j$, and then relabeling via the unique increasing bijection $\{1,\dots,2r+1\}\setminus\{\ell_i,\ell_{i+1}\}\to\{1,\dots,2r-1\}$.
  \end{itemize}
  Figure~\ref{fig:liu_shrink} shows an example of this procedure.
  
  Note that the content of $(T',\mathrm{dec}',\mathrm{lr}')$ is $(c_1,\dots,c_{i-2},c_{i-1}+(j-l_{i}-1),c_{i+1}+(l_{i+1}-j-1),c_{i+2},\dots,c_r)$. Also, all bilabeled trees with this content are obtained in the manner described above. 
\begin{figure}
  \includegraphics[scale=0.8]{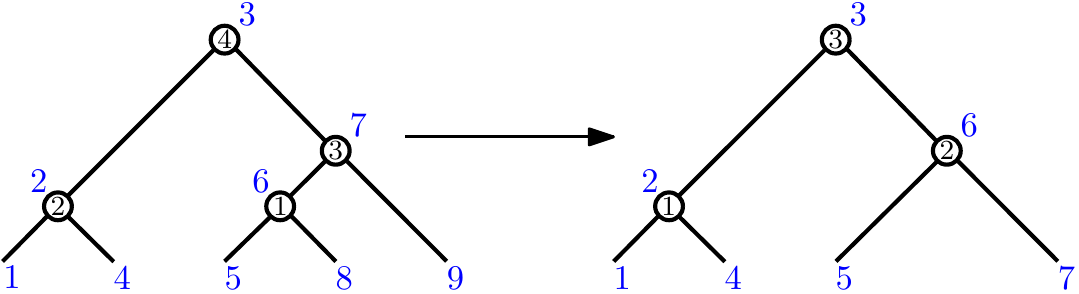}
  \caption{An example of $(T,\mathrm{dec},\mathrm{lr})\to (T',\mathrm{dec}',\mathrm{lr}') $.}
  \label{fig:liu_shrink}
  \end{figure}
 We thus get
 \begin{align}
 \label{eq:liu_recurrence}
 A_c=A_{(c_1+c_2-1,c_3,\dots,c_r)}+A_{(c_1,\dots,c_{r-2},c_{r-1}+c_r-1)}
 +\sum_{i=2}^{r-1}\sum_{t=0}^{c_i-1}A_{(c_1,\dots,c_{i-2},c_{i-1}+t,c_{i+1}+c_i-1-t,c_{i+2}\dots,c_r)},
  \end{align}
which is precisely Liu's recurrence \cite[p. 8]{Liu16}.
Using the same approach, one can in fact obtain a bijection between bilabeled trees with content $c$ and  Liu's $c$-permutations.

\section{Concluding remarks}
We conclude this article with some brief remarks that further demonstrate the combinatorial richness of the $A_c$ and hopefully motivate the reader to investigate.

\begin{enumerate}
\item The statement of Theorem~\ref{th:q_list}\ref{it4} has the mildly unattractive aspect in that both the ordinary factorial and the $q$-factorial show up. In fact, appealing to Theorem~\ref{th:q_list}\ref{it7}, one can establish another $q$-analogue:
\begin{align}
\label{eq:pfr}
\sum_{c\in \wc{r}{}}\frac{A_{c_1,\dots,c_{r}}(q)}{\qfact{c_1}\cdots \qfact{c_r}}=\sum_{\substack{c\in \wc{r}{}\\\forall i\leq r,c_1+\cdots+c_i\geq i}}\frac{\qfact{r}}{\qfact{c_1}\cdots \qfact{c_r}}=\sum_{f \in \mathrm{PF}(r)}q^{\inv(f)}.
\end{align}
Here $\mathrm{PF}(r)$ is the set of parking functions of length $r$. 
The first equality in~\eqref{eq:pfr} is simply invoking Theorem~\ref{th:q_list}\ref{it7}, whereas the second follows since the $q$-multinomial coefficient $\frac{\qfact{r}}{\qfact{c_1}\cdots \qfact{c_r}}$ for $c\in \wc{r}{}$ satisfying $c_1+\cdots+c_i\geq i$ for all $i\leq r$ tracks the inversion (equivalently the major index) statistic over parking functions with content $c$. Note that each individual summand on the left-hand side in~\eqref{eq:pfr} is not necessarily a polynomial in $q$. Furthermore, one could argue that the appearance of parking functions in the last step is contrived.\\ We say more on this matter in upcoming work that describes another (multivariate) perspective on the $A_c(q)$ wherein parking procedures come up organically. In that same work the combinatorial interpretation for $A_c(q)$ as counting bilabeled trees, obtained in Sections~\ref{sec:volumes} and \ref{sec:A_c bilabeled} via geometric means, will be derived purely algebraically.

\item
As mentioned earlier, the authors' inspiration to introduce, and study, the $A_c(q)$ came from Schubert calculus \cite{NT20,NT21}. In particular, the Schubert class expansion of the cohomology class of the type $A$ permutahedral variety naturally involves mixed Eulerian numbers once the connection between ordinary divided symmetrization and coefficient extraction Klyachko's algebra \cite{Kly85} is made. 
Since the latter algebra also arises as the cohomology ring of a regular nilpotent Hessenberg variety known as the \emph{Peterson variety}, it is not surprising that mixed Eulerian numbers arise in that context.\\
We make note of their occurrence in two recent works. 
Goldin--Gorbutt~\cite{GoGo20} compute structure coefficients in the \emph{Peterson Schubert basis} and give explicit expression for them in certain cases. In particular, Corollary 2 in loc. cit. describes mixed Eulerian numbers $A_c(1)$ where $c$ is a composition of the form $(0^p1^q2^r1^s0^t)$ for the appropriate $p,q,r,s,$ and $t$, even though these numbers are not explicitly identified as mixed Eulerian numbers in said work.
More generally, several structural constants in their work are products of $A_c$ where the $c$ satisfy $c_i\leq 2$. The type $A$ story in this context is also present in~\cite{Hor21}.

\item Finally, other subfamilies of $A_c(q)$ of interest are currently being investigated by Solal Gaudin, a student of the first author, as part of his thesis: as an example, weak compositions $c$ whose support has size two, that is particle configurations with two piles. Note that these are not in general elements of the two subfamilies studied in Section~\ref{sec:examples}. 

\end{enumerate}

\section*{Acknowledgements}
V. T. is grateful to Lauren Williams for helpful correspondence concerning Bruhat interval~polytopes. We are also grateful to Jang Soo Kim for several comments on the manuscript and for enlightening discussions.
\bibliographystyle{hplain}
\bibliography{Biblio_remixed}

\begin{thebibliography}{10}

\bibitem{Carlitz}
L.~Carlitz.
\newblock {$q$}-{B}ernoulli and {E}ulerian numbers.
\newblock {\em Trans. Amer. Math. Soc.}, 76:332--350, 1954.

\bibitem{Dworkin}
M.~Dworkin.
\newblock An interpretation for {G}arsia and {R}emmel's {$q$}-hit numbers.
\newblock {\em J. Combin. Theory Ser. A}, 81(2):149--175, 1998.

\bibitem{GarsiaRemmel}
A.~M. Garsia and J.~B. Remmel.
\newblock {$Q$}-counting rook configurations and a formula of {F}robenius.
\newblock {\em J. Combin. Theory Ser. A}, 41(2):246--275, 1986.

\bibitem{GoGo20}
R.~Goldin and B.~Gorbutt.
\newblock A positive formula for type {$A$} {P}eterson {S}chubert calculus.
\newblock {\em La Mathematica}, 2022.

\bibitem{Haglund}
J.~Haglund.
\newblock {$q$}-rook polynomials and matrices over finite fields.
\newblock {\em Adv. in Appl. Math.}, 20(4):450--487, 1998.

\bibitem{HarHor18}
M.~Harada, T.~Horiguchi, M.~Masuda, and S.~Park.
\newblock The volume polynomial of regular semisimple {H}essenberg varieties
  and the {G}elfand-{Z}etlin polytope.
\newblock {\em Proc. Steklov Inst. Math.}, 305:318--344, 2019.

\bibitem{Hor21}
T.~Horiguchi.
\newblock Mixed {E}ulerian numbers and {P}eterson {S}chubert calculus, 2021,
  arXiv:2104.14083.

\bibitem{KST12}
V.~A. Kirichenko, E.~Yu. Smirnov, and V.~A. Timorin.
\newblock Schubert calculus and {G}elfand-{T}setlin polytopes.
\newblock {\em Uspekhi Mat. Nauk}, 67(4(406)):89--128, 2012.

\bibitem{Kir10}
V.~Kiritchenko.
\newblock Gelfand-{Z}etlin polytopes and flag varieties.
\newblock {\em Int. Math. Res. Not. IMRN}, (13):2512--2531, 2010.

\bibitem{Kly85}
A.~A. Klyachko.
\newblock Orbits of a maximal torus on a flag space.
\newblock {\em Funktsional. Anal. i Prilozhen.}, 19(1):77--78, 1985.

\bibitem{Kog00}
M.~Kogan.
\newblock {\em Schubert geometry of flag varieties and {G}elfand-{C}etlin
  theory}.
\newblock ProQuest LLC, Ann Arbor, MI, 2000.
\newblock Thesis (Ph.D.)--Massachusetts Institute of Technology.

\bibitem{LP20}
T.~Lam and A.~Postnikov.
\newblock Polypositroids, 2020, arXiv:2010.07120.

\bibitem{Liu16}
G.~Liu.
\newblock Mixed volumes of hypersimplices.
\newblock {\em Electron. J. Combin.}, 23(3):Paper 3.19, 19 pages, 2016.

\bibitem{Man01}
L.~Manivel.
\newblock {\em Symmetric functions, {S}chubert polynomials and degeneracy
  loci}, volume~6 of {\em SMF/AMS Texts and Monographs}.
\newblock American Mathematical Society, Providence, RI; Soci\'{e}t\'{e}
  Math\'{e}matique de France, Paris, 2001.
\newblock Translated from the 1998 French original by John R. Swallow, Cours
  Sp\'{e}cialis\'{e}s [Specialized Courses], 3.

\bibitem{Mit20}
K.~Mittelstaedt.
\newblock A stochastic approach to {E}ulerian numbers.
\newblock {\em Amer. Math. Monthly}, 127(7):618--628, 2020.

\bibitem{NT21}
P.~Nadeau and V.~Tewari.
\newblock A $q$-analogue of an algebra of {K}lyachko and {M}acdonald's reduced
  word identity, 2021, arXiv:2106.03828.

\bibitem{NT20}
P.~Nadeau and V.~Tewari.
\newblock The permutahedral variety, mixed {E}ulerian numbers, and principal
  specializations of {S}chubert polynomials.
\newblock {\em Int. Math. Res. Not. IMRN}, 2021, to appear.

\bibitem{NTchromatic}
P.~Nadeau and V.~Tewari.
\newblock Down-up algebras and chromatic symmetric functions, 2022,
  arXiv:2208.04175.

\bibitem{Pos09}
A.~Postnikov.
\newblock Permutohedra, associahedra, and beyond.
\newblock {\em Int. Math. Res. Not. IMRN}, (6):1026--1106, 2009.

\bibitem{St97}
R.~P. Stanley.
\newblock {\em Enumerative combinatorics. {V}ol. 1}, volume~49 of {\em
  Cambridge Studies in Advanced Mathematics}.
\newblock Cambridge University Press, Cambridge, 1997.
\newblock With a foreword by Gian-Carlo Rota, Corrected reprint of the 1986
  original.

\bibitem{St99}
R.~P. Stanley.
\newblock {\em Enumerative combinatorics. {V}ol. 2}, volume~62 of {\em
  Cambridge Studies in Advanced Mathematics}.
\newblock Cambridge University Press, Cambridge, 1999.
\newblock With a foreword by Gian-Carlo Rota and appendix 1 by Sergey Fomin.

\bibitem{Stanley-Pitman}
R.~P. Stanley and J.~Pitman.
\newblock A polytope related to empirical distributions, plane trees, parking
  functions, and the associahedron.
\newblock {\em Discrete Comput. Geom.}, 27(4):603--634, 2002.

\bibitem{TW15}
E.~Tsukerman and L.~Williams.
\newblock Bruhat interval polytopes.
\newblock {\em Adv. Math.}, 285:766--810, 2015.

\end{thebibliography}

\end{document}